\documentclass[reqno]{amsart}

\usepackage{latexsym,amsfonts}
\usepackage{amssymb, amscd,fge}
\usepackage{amsmath,amsthm}          
\usepackage{euscript}                
\usepackage{epsfig}
\usepackage[all]{xy}
\usepackage{mathrsfs}
\usepackage{calrsfs}

\usepackage{mathtools}

\usepackage[bookmarks,
bookmarksnumbered,%
 colorlinks=true,%
 linkcolor=blue,%
 citecolor=red,%
 filecolor=blue,%
 urlcolor=blue,%
]
{hyperref}

\theoremstyle{plain} \numberwithin{equation}{section}
\newtheorem{main}{Theorem}

\newtheorem{thm}{Theorem}[section]
\newtheorem{cor}[thm]{Corollary}
\newtheorem{prop}[thm]{Proposition}

\newtheorem{lemma}[thm]{Lemma}

\newtheorem{coro}[main]{Corollary}

\theoremstyle{definition}
\newtheorem{remark}{Remark}[section]

\newtheorem{defn}[remark]{Definition}
\newtheorem{ex}[remark]{Example}

\newtheorem{rmk}[thm]{Remark}

\newcommand{\bi}{\begin{itemize}}
\newcommand{\ei}{\end{itemize}}
\newcommand{\bp}{\begin{proof}}
\newcommand{\ep}{\end{proof}}
\newcommand{\ddbar}{\partial \bar \partial}

\def\dim{\mbox{dim}}
\def\ra{\rightarrow}
\def\cal{\mathcal}

\def\CC{\mathbb{C}}
\def\PP{\mathbb{P}}
\def\QQ{\mathbb{Q}}
\def\ZZ{\mathbb{Z}}
  
\def\RR{\mathbb{R}}

\def\HH{\mathbb{H}}

\def\OO{\cal O}

\def\s-{\setminus}

\begin{document}

\title[Remarks on astheno-K\"ahler manifolds]{Remarks on astheno-K\"ahler manifolds, 
Bott-Chern and Aeppli cohomology groups}

\author[Chiose]{Ionu\c{t} Chiose}

\address{
	Institute of Mathematics of the Romanian Academy,  P.O. Box 1-764, Bucharest 014700,  Romania}
	
	\email{Ionut.Chiose@imar.ro}

\author[R\u asdeaconu]{Rare\c s R\u asdeaconu}

\address{        
        Department of Mathematics, 1326 Stevenson Center, Vanderbilt University, Nashville, TN, 37240, USA}
        
        \email{rares.rasdeaconu@vanderbilt.edu}

\date{\today}



\dedicatory{Dedicated to the memory of Professor Jean-Pierre Demailly.}

\begin{abstract}
We provide a new cohomological obstruction to the existence of astheno-K\"ahler metrics, 
and study relevant examples.
\end{abstract}

\maketitle

\thispagestyle{empty}

%\tableofcontents

\section{Introduction}

\medskip

Let $(M,J,g)$ be a Hermitian manifold of real dimension $2n$ and  
$\displaystyle
\omega(\cdot,\cdot) = g(J\cdot,\cdot)
$ 
its fundamental $2$-form.  If $\omega$ is $d$-closed, then the metric $g$ is called K\"ahler, 
and a complex manifold carrying such a metric is called a K\"ahler manifold. K\"ahler manifolds 
exist in abundance and satisfy well-documented remarkable cohomological properties. For 
example, they are formal and the $\ddbar$-lemma \cite{dgms} holds, the Hodge symmetry 
is satisfied, and the Hodge-Fr\"olicher spectral sequence degenerates at the first page. 
Furthermore, a result of Harvey-Lawson \cite{hl} states that a compact complex manifold 
carries a K\"ahler metric if and only if it carries no positive $(1,1)$-components of boundaries. 
Nevertheless, the failure of such cohomological properties obstructs the existence of K\"ahler 
metrics, and many examples of non-K\"ahler manifolds are present in the literature. It is only 
natural to impose weaker conditions on the fundamental $2$-form $\omega.$ Such conditions 
have been oftentimes considered and studied, and many of them involve the closure with 
respect to the $\ddbar$-operator of the $(k, k)$-form $\omega^k.$ However, the presence 
of special classes of Hermitian metrics hardly ever imposes good cohomological behavior.

\smallskip

A Hermitian metric satisfying 
$$
\ddbar \omega^{n-2}=0
$$
is called an {\it{astheno-K\"ahler metric}}. Such metrics were introduced by Jost and 
Yau in their study of the Hermitian harmonic maps, and  used to prove an extension 
of Siu's Rigidity Theorem to non-K\"ahler manifolds \cite[Theorem 6]{jy}. Later, Li, 
Yau and Zheng found other interesting applications, such as a generalization to higher 
dimension of Bogomolov's Theorem on class $VII_0$ surfaces \cite[Corollary 3]{lyz}, 
while Carlson and Toledo used them in \cite{ct} to obtain results on the fundamental 
groups of class $VII$ surfaces. Several construction methods of astheno-K\"ahler 
metrics are currently known (see Section \ref{eao}). However, only few obstructions to 
their existence are known. The obstructions are derived from an observation of Jost and 
Yau \cite{jy}, who noticed  that on a manifold carrying an astheno-K\"ahler metric every 
holomorphic $1$-form is closed. This remark was generalized by Fino, Grantcharov and 
Vezzoni \cite{fgv} who found a Harvey-Lawson type criterion for astheno-K\"ahler metrics: 
if a compact complex manifold admits a weakly positive, $\ddbar$-exact, non-vanishing 
$(2,2)$-current, then it cannot carry an astheno-K\"ahler metric. 

\smallskip

The aim of this article is to exhibit and study a new relation between the Bott-Chern 
and Aeppli cohomologies of a compact complex manifold which appears at the level 
of $(0,1)$-forms in the presence of an astheno-K\"ahler metric. In Section \ref{g-arg} 
the following result is proved.
 \begin{main}
 \label{c-ineq}
On a compact astheno-K\"ahler manifold $M$, the 
following inequalities hold:
\begin{equation}
\label{ineq}
h^{0,1}_{BC}(M)\leq h_A^{0,1}(M)\leq h^{0,1}_{BC}(M)+1.
\end{equation}
\end{main}

\smallskip

Such a result yields an obstruction to the existence of astheno-K\"ahler metrics on 
a given complex manifold. In Section \ref{applications} we test this obstruction against 
two classes of non-K\"ahler manifolds, the Nakamura and the Oeljeklaus-Toma 
manifolds, confirming that they cannot carry astheno-K\"ahler metrics.

Several applications are obtained by appealing to two results of independent interest  
regarding the  Bott-Chern and Aeppli cohomologies  of compact complex manifolds 
which are proved in Section \ref{BCcoh}. The first one is a very weak form of a 
K\"unneth formula. It should be pointed out that while a K\"unneth formula 
for the Dolbeault cohomology is available \cite{gh}, a similar formula is not known for the 
Bott-Chern and Aeppli cohomology theories. However, a weak 
form holds for $(0,1)$-forms, which suffices for the applications considered here.
 \begin{main}
 \label{weak-kunneth}
If $X$ and $Y$ are compact complex manifolds, then
\begin{align}
\label{wk}
h^{0,1}_{BC}(X\times Y)= h^{0,1}_{BC}(X)+h^{0,1}_{BC}(Y)\\ \notag
h^{0,1}_{A}(X\times Y)\geq h^{0,1}_{A}(X)+h^{0,1}_{A}(Y).
\end{align}
\end{main}

This  result is used in combination with Theorem \ref{c-ineq} in Section \ref{applications} 
to study the existence of astheno-K\"ahler metrics on Cartesian products. As a particular 
case, the following result is obtained:

\begin{main}
\label{cart-surfaces}
A Cartesian product of two compact complex surfaces admits an astheno-K\"ahler 
metric if and only if at least one of the surfaces admits a K\"ahler metric.
\end{main}
It is noticed in Remark \ref{non-jy} that to prove that a Cartesian of two non-K\"ahler 
surfaces does not carry an astheno-K\"ahler metric, the Jost-Yau obstruction provides 
no relevant information.

\smallskip

A second result of independent interest obtained in Section \ref{BCcoh} concerns the 
Bott-Chern and Aeppli cohomologies for nilmanifolds with nilpotent complex structure: 
\begin{main}
\label{torality}
Let $M$ be non-K\"ahler nilmanifold equipped with a nilpotent complex structure. 
If $h_{A}^{0,1}(M)=h_{BC}^{0,1}(M),$ then $M$ is a compact complex torus.
\end{main}
As an immediate application of Theorems \ref{c-ineq} and \ref{torality} we obtain:
\begin{coro}
\label{nil-nilp}
Let $M$ be non-K\"ahler nilmanifold equipped with a nilpotent complex structure. 
If $M$ admits an astheno-K\"ahler metric then $h_{A}^{0,1}(M)=h_{BC}^{0,1}(M)+1.$
\end{coro}
In small dimensions, there are known examples of nilmanifolds with nilpotent complex 
structure saturating the upper bound and carrying invariant astheno-K\"ahler metric 
\cite{ft2,fgv, rt, st}. In Section \ref{eao} it is noticed that there exist astheno-K\"ahler 
metrics on nilmanifolds with nilpotent complex structure of arbitrary dimension 
(see Theorem \ref{ex-nilp}).

\smallskip

Conversely, one may ask if a compact complex manifold saturating either one 
of the bounds in (\ref{ineq}) carries an astheno-K\"ahler metric. For the upper 
bound, it is already known that in dimension six the answer is negative. For 
example, the Iwasawa manifold has $h^{0,1}_{BC}=2$ and $h^{0,1}_A=3$ 
\cite{angella}, and yet it does not carry astheno-K\"ahler metrics, as it supports 
non-closed holomorphic $1$-forms, contradicting the Jost-Yau criterion. More 
examples of nilmanifolds satisfying similar properties can be found in dimension 
six from the classification of the SKT structures in \cite{fps} and the Bott-Chern 
cohomology computations in \cite{luv}.  For the upper bound, in Section 
\ref{examples} one more class of relevant examples is indicated.
\begin{main}
\label{vaisman}
Any compact Vaisman manifold of dimension at least three satisfies 
$h^{0,1}_A=h^{0,1}_{BC}+1$ and carries no astheno-K\"ahler metric. 
\end{main}
As a consequence, an old conjecture of Li, Yau and Zheng \cite[page 108]{lyz} 
is confirmed:
\begin{coro}
\label{lyz-answer}
There exists no astheno-K\"ahler metrics on similarity Hopf manifolds of 
dimension at least three.
\end{coro}

 \smallskip
 
For the lower bound, in Section \ref{examples} two examples are exhibited, a 
solvmanifold of complex dimension $2n,\,n\geq 3$ and the Fujiki class $\mathcal C$ 
non-K\"ahler manifolds of dimension three, which do not carry astheno-K\"ahler 
metrics while saturating the lower bound in (\ref{ineq}). In particular, one can see 
that the class of astheno-K\"ahler manifolds is not invariant under modifications 
(see Corollary \ref{modifications}).
The authors are not aware of any example of a non-K\"ahler, compact, complex, 
astheno-K\"ahler manifold satisfying $h^{0,1}_{BC}=h^{0,1}_{A}.$

%%%%%%%%%%%%%%%%%%%%%%%%%%%%%%%%%%%%%%%%%%%%%%%%%%%
%%%%%%%%%%%%%%%%%%%%%%%%%%%%%%%%%%%%%%%%%%%%%%%%%%%
%%%%%%%%%%%%%%%%%%%%%%%%%%%%%%%%%%%%%%%%%%%%%%%%%%%
%%%%%%%%%%%%%%%%%%%%%%%%%%%%%%%%%%%%%%%%%%%%%%%%%%%

\section{Existence and obstructions}
\label{eao}

\begin{defn}
Let $X$ be an $n$-dimensional complex manifold.  $X$ is called $p$-pluriclosed 
if it admits a Hermitian metric $g$ whose fundamental form $\omega$ satisfies 
 $$
 i\partial\bar\partial \omega^{n-p}=0.
 $$
\end{defn}

Gauduchon  showed that $1$-Gauduchon metrics always exist \cite{gauduchon}, a result 
with widespread implications, and such metrics are known as  {\it{Gauduchon metrics}}. 
The $2$-Gauduchon metrics appeared for the first time in the work of Jost and Yau \cite{jy} 
under the name of {\it{astheno-K\"ahler metrics}}, and used in a variety of applications 
\cite{ct, jy,lyz}. The $(n-1)$-pluriclosed metrics were introduced by Bismut \cite{bismut} 
and have received a lot of attention in the recent years under different names, such as  
{\it strongly K\"ahler with torsion} (SKT) or {\it pluriclosed metrics}.

%%%%%%%%%%%%%%%%%%%%%%%%%%%%%%%%%%%%
%%%%%%%%%%%%%%%%%%%%%%%%%%%%%%%%%%%%
%%%%%%%%%%%%%%%%%%%%%%%%%%%%%%%%%%%%

\subsection{Existence of astheno-K\"ahler metrics on nilmanifolds} 
\label{e-ak-nil}

There are currently a few general methods to construct non-K\"ahler 
astheno-K\"ahler metrics. The most powerful results are obtained on 
complex nilmanifolds. 

\begin{defn}
A complex nilmanifold  $\Gamma\backslash G$ is a quotient of a 
simply-connected, connected nilpotent Lie group $G$ endowed with a 
left invariant integrable almost complex structure by a lattice  
$\Gamma\subset G$ of maximal rank.
\end{defn}

A nilmanifold  $\Gamma\backslash G$ inherits its complex structure from 
that of $G$ by passing to the quotient. 

In complex dimension three, where the notions of SKT and astheno-K\"ahler 
metrics coincide, we have the following result of Fino, Parton and Salamon: 

\begin{thm}[Theorem 1.2 \cite{fps}]
\label{fps-structure}
Let $M=\Gamma\backslash G,$ be a (real) six-dimensional nilmanifold with 
an invariant complex structure $J.$ Then the SKT condition is satisfied if and 
only if $M$ has a basis $\alpha_i, i=1,2,3$ of $(1,0)$-forms such that:
\begin{equation}
\label{3eq}
\begin{cases}
d\alpha_1 = 0\\
d \alpha_2=0 \\
d\alpha_3 = A\bar \alpha_1\wedge \alpha_2 + 
B\bar \alpha_2\wedge \alpha_2 +
C\alpha_1\wedge \bar \alpha_1 + 
D\alpha_1\wedge \bar \alpha_2 + 
E\alpha_1\wedge \alpha_2,
\end{cases}
\end{equation}
where  $A, B, C, D, E$ are complex numbers such that
$$
|A|^2 + |D|^2 + |E|^2 + 2\mathfrak{Re}(\bar BC) = 0. 
$$
\end{thm}

\begin{rmk}
There are $18$ isomorphism classes of the underlying nilpotent Lie algebras 
\cite{sal}, out of which one is the compact complex three-dimensional torus, 
the only K\"ahler example, and only four of them are non-K\"ahler and admit 
SKT metrics \cite[Theorem 3.2]{fps}. The interested reader is also referred to 
\cite{ugarte} for a different proof of these result. 
\end{rmk}

To construct examples of left invariant astheno-K\"ahler metrics on complex 
nilmanifolds in arbitrary dimension it is useful to impose suitable conditions on 
the structure equations.  Generalizing Theorem \ref{fps-structure}, examples 
of astheno-K\"ahler metrics with interesting properties were found in the recent 
years on nilmanifolds with nilpotent complex structure \cite{ft2,fgv,rt,st,lu}.

\smallskip

Let $G$ be a nilpotent group of dimension $2n,$ with Lie algebra $\mathfrak g.$ 
Suppose $G$ carries a left invariant integrable almost complex structure $J.$ 
An ascending series $\{\mathfrak a_l; l\geq 0\}$ compatible with $J$ is defined 
inductively by 
\begin{equation*}
\mathfrak a_0=0,\quad \mathfrak a_l=
\{X\in \mathfrak g\,|\,[X,\mathfrak g] \subseteq \mathfrak a_{l-1} 
\,{\text{and}}\, [JX,\mathfrak g]\subseteq\mathfrak a_{l-1}\}, 
\quad l\geq 1.
\end{equation*}
It is easy to verify that the $\mathfrak a_l$ is an ideal of $\mathfrak g$ and a 
complex subspace of $\mathfrak g.$  Moreover, 
$\mathfrak a_l\subseteq  \mathfrak a_{l+1},$ 
for each $l\geq 0,$ and if $\mathfrak a_l = \mathfrak a_{l+1}$ for some 
$l\geq 0,$ then $\mathfrak a_r = \mathfrak a_l$ for all $r\geq l$ \cite{cfgu}.

\begin{defn}[\it{cf.}, \cite{cfgu}]
Let $G$ be a $2n$-dimensional simply-connected connected nilpotent Lie 
group with Lie algebra $\mathfrak g,$ and equipped with a left invariant 
integrable almost complex structure $J.$
\begin{itemize}
\item[ 1)] We shall say that $J$ is a nilpotent complex structure if 
$\mathfrak a_k = \mathfrak g$ for some $k > 0.$
\item[ 2)] Furthermore, if $J$ is a nilpotent, left-invariant complex structure on 
$G,$ and $\Gamma$  is a co-compact lattice of $G$, we shall say that the 
compact nilmanifold $\Gamma\backslash G$ with the complex structure induced 
by  $J$ has a nilpotent complex structure.
\end{itemize}
\end{defn}
We recall next the following characterization of the nilpotent complex structure:
\begin{thm}[Theorems 12 and 13, \cite{cfgu}]
\label{equiv-nil}
Let $G$ be a $2n$-dimensional simply-connected connected nilpotent Lie group 
with Lie algebra $\mathfrak g,$ and equipped with a left invariant integrable almost
complex structure $J.$ The complex structure $J$ is nilpotent if and only if there 
exists a (complex) basis $\{\alpha_i, \, i\leq i\leq n \}$ of left invariant forms of type 
$(1,0)$ such that the structure equations of $G$ are of the form 
\begin{equation}
\label{structure}
d \alpha_i=\sum_{j<k<i} A^i_{j,k}\, \alpha_j\wedge \alpha_k+
\sum_{j,k<i} B^i_{j,k}\, \alpha_j\wedge \bar \alpha_k, 
\quad i=1,\dots n, 
\end{equation}
where $A^i_{j,k}$ and $B^i_{j,k}$ are constants. Conversely, the structure equations 
(\ref{structure}) define a simply-connected connected nilpotent Lie group $G$ with 
nilpotent left invariant complex structure.
\end{thm}

\begin{rmk}
\label{para}
The complex parallelizable nilmanifolds are precisely those nilmanifolds with nilpotent 
complex structure for which the coefficients $B^i_{jk}$ in (\ref{structure}) vanish, in which 
case all of the forms $\alpha_i,\, i=1,\dots,n$ are holomorphic.
\end{rmk}

The following result exhibits astheno-K\"ahler metrics on compact nilmanifolds equipped 
with a nilpotent complex structure, generalizing  the construction in \cite[Theorem 2.7]{ft2} 
to arbitrary dimension.

\begin{thm}
\label{ex-nilp} 
Let $G$ be the simply-connected nilpotent Lie group with nilpotent complex structure given 
by the $(1,0)$-forms $\alpha_i, i=1,\dots, n,$ satisfying the structure equations:
\begin{equation*}
\label{neq}
\begin{cases}
d\alpha_i = 0, i=1,\dots, n-1,\\
d\alpha_n = \sum_{i<j<n} A_{i,j}\, \alpha_i\wedge \alpha_j+
\sum_{k,l<n} B_{i,j}\, \alpha_k\wedge \bar \alpha_l,
\end{cases}
\end{equation*}
where $A_{ij}, B_{k,l}\in \CC.$ Then $G$ carries an invariant astheno-K\"ahler metric if 
$$
\sum_{i<j<n}|A_{ij}|^2+\sum_{i,j<n, i\neq j}|B_{ij}|^2+
2\mathfrak{Re}(\sum_{i<j<n}B_{ii}\bar B_{jj})=0.
$$
Furthermore, for every $n\geq 3,$ there exist compact nilmanifolds with nilpotent complex 
structures of complex dimension $n$ carrying an invariant astheno-K\"ahler metric.
\end{thm}

\begin{proof} As in \cite[Theorem 2.7]{ft2}, it is enough to find sufficient conditions satisfied 
by the coefficients on the structure equations (\ref{structure}) such that the diagonal metric 
$$
g=\frac12\sum_{i=1}^n \alpha_i\otimes \bar \alpha_i 
+\bar \alpha_i\otimes \alpha_i 
$$
is astheno-K\"ahler.  In other words, we will require that 
$\ddbar \omega^{n-2}=0,$ where 
$$
\omega=\frac{i}{2}\sum_{i=1}^n\alpha_i\wedge\bar \alpha_i
$$ 
is the fundamental form of the metric $g.$ To adapt to 
the structure equations, it is convenient to write $\omega$ as
$$
\omega=\frac{i}{2}\left( \omega_0+\alpha_n\wedge \bar \alpha_n\right),
$$ 
where 
$\displaystyle \omega_0=\sum_{i=1}^{n-1} 
\alpha_i\wedge\bar \alpha_i.$ 
An immediate computation shows that 
\begin{equation}
\label{power}
\omega^{n-2}=\left(\frac{i}{2}\right)^{n-2}\left( \omega_0^{n-2}+
(n-2)\omega_0^{n-3}\wedge \alpha_n\wedge \bar \alpha_n\right),
\end{equation}
while, for every $p>0$
$$
\omega_0^p= \binom{n-1}{p}\sum_{i_1<\cdots<i_p<n}
\alpha_{i_1\bar i_1\cdots i_p\bar i_p},
$$
where $\alpha_{i_1\bar i_1\cdots i_p\bar i_p}$ is a short 
notation for $\alpha_{i_1}\wedge \bar \alpha_{i_1}
\wedge\cdots\wedge \alpha_{i_p}\wedge \bar\alpha_{i_p},$ 
a notation which we will adopt henceforth.

Notice now from the structure equations (\ref{neq}) that 
$\partial \alpha_{i}=\bar\partial \alpha_i=
\partial\bar \alpha_{i}=\bar\partial \bar \alpha_i=0$ for 
every $1=1,\dots, n-1,$ while 
\begin{equation}
\label{an}
\partial \alpha_n=\sum_{i<j<n}A_{ij}\alpha_{ij}\quad{\text{and}}\quad 
\bar\partial \alpha_n= \sum_{i,j<n}B_{ij}\alpha_{i\bar j}.
\end{equation}
Hence, we have  $\ddbar \alpha_n=\ddbar\bar \alpha_n=0,$ and  
$\partial \omega_0^k=\bar\partial \omega_0^k=0$ for every $k>0.$ In 
particular, from (\ref{power}), we see now that 
\begin{align*}
\ddbar \omega^{n-2}=&\, \left(\frac{i}{2}\right)^{n-2}(n-2)
\omega_0^{n-3}\wedge \ddbar \alpha_{n\bar n}\\
=&\, c_n \left(\sum_{i_1<\cdots<i_{n-3}<n} 
\alpha_{i_1\bar i_1\cdots i_p\bar i_{n-3}}\right)\wedge 
\left( \bar\partial \alpha_n\wedge \partial \bar \alpha_n - 
\partial \alpha_n\wedge \bar\partial \bar \alpha_n\right)\\
=&\, c_n \left(\sum_{i<j<n}|A_{ij}|^2+\sum_{i,j<n, i\neq j}|B_{ij}|^2
+2\mathfrak{Re}(\sum_{i<j<n}B_{ii}\bar B_{jj})\right)
\alpha_{1\bar 1\cdots (n-1)\overline{(n-1)}}.
\end{align*}
where $c_n$ is a non-zero constant depending on $n.$

Imposing now the condition $\ddbar \omega^{n-2}=0$ and using (\ref{an}), 
we find that the metric $g$ is astheno-K\"ahler if and only if 
\begin{equation}
\label{ast-cond}
\sum_{i<j<n}|A_{ij}|^2+\sum_{i,j<n, i\neq j}|B_{ij}|^2+
2\mathfrak{Re}(\sum_{i<j<n}B_{ii}\bar B_{jj})=0.
\end{equation}

The conclusion of the theorem follows by noticing that (\ref{ast-cond}) 
admits solutions  with $A_{ij}, B_{ij}\in \QQ[i],$ for example one can take  
$A_{ij}=1,$ for all $1\leq i<j<n,\, B_{ij}=\frac12$ for all $1\leq i\neq j<n,$ 
and $B_{ii}=i.$ In this case, it is guaranteed by Mal\v cev's theorem \cite{malcev} 
that we can find a maximal rank lattice $\Gamma\subset G$ such that 
$M =\Gamma\backslash G$  is a compact nilmanifold with nilpotent complex 
structure. The induced metric will automatically be astheno-K\"ahler.
\end{proof}

%%%%%%%%%%%%%%%%%%%%%%%%%%%%%%%%%%%%%%
%%%%%%%%%%%%%%%%%%%%%%%%%%%%%%%%%%%%%%
%%%%%%%%%%%%%%%%%%%%%%%%%%%%%%%%%%%%%%

\subsection{Other constructions of astheno-K\"ahler metrics}
\label{other-ak}

For convenience, we will mention a few other sources of astheno-K\"ahler 
manifolds. Products of Sasakian manifolds carry astheno-K\"ahler metrics, 
as observed by Matsuo \cite{matsuo}. In particular, the Calabi-Eckmann 
manifolds carry such metrics. This method was generalized by Fino, Grantcharov 
and Vezzoni. In \cite{fgv}, who constructed  astheno-K\"ahler metrics on 
torus bundles over a K\"ahler base. This technique recovers some of the known 
astheno-K\"ahler metrics on  nilmanifolds with nilpotent complex structure, 
Matsuo's metrics, and also produces non-homogeneous ones. We will not 
add the details, as we will only briefly mention the Calabi-Eckmann case. 
For $n > 3$ other examples of astheno-K\"ahler manifolds have been found 
by Fino and Tomassini \cite{ft2} via the twist construction \cite{swann}.

\smallskip

Finally, according to Fino and Tomassini \cite{ft1,ft2}, starting from previously 
known examples of SKT and astheno-K\"ahler manifolds one can construct 
new ones:

\begin{thm}[Fino, Tomassini] 
\label{fino-tomassini}
Let $M$ be a complex manifold, $Y\subset M$ is  a compact complex submanifold, 
and $\widetilde M_Y$ the blowing-up  of $M$ along $Y.$
\begin{itemize} 
\item[ 1)] If $M$ admits a pluriclosed metric then $\widetilde M_Y$  admits 
a pluriclosed metric.
\item[ 2)] If $M$ admits an astheno-K\"ahler metric $\omega$ such that
\begin{equation}
\label{ft-condition}
\ddbar \omega=0\quad{\text{and}}\quad \ddbar\omega^2=0,
\end{equation}
then $\widetilde M_Y$  admits an astheno-K\"ahler metric satisfying (\ref{ft-condition}), 
too.
\end{itemize}
\end{thm}

%%%%%%%%%%%%%%%%%%%%%%%%%%%%%%%%%%%%%%
%%%%%%%%%%%%%%%%%%%%%%%%%%%%%%%%%%%%%%
%%%%%%%%%%%%%%%%%%%%%%%%%%%%%%%%%%%%%%

\subsection{Known obstructions to the existence of astheno-K\"ahler metrics} 
\label{old-obs-ak}

The presence of astheno-K\"ahler metric was noticed to force strong conditions 
on holomorphic $1$-forms by Jost and Yau.
 
 \begin{lemma}[Jost, Yau \cite{jy}]
\label{jy}
Let $X$ be a compact astheno-K\"ahler manifold. Then every holomorphic 
$1$-form on $X$ is closed.
\end{lemma}

\begin{proof} Let $\omega$ satisfy the conditions of the definition. 
Let $\phi$ be a holomorphic $1$-form,
i.e., $\bar\partial \phi=0.$ Then
$$
0\leq \int \partial \phi\wedge \bar \partial \bar\phi \wedge \omega^{n-2} = 
\int  \phi\wedge \bar\phi \wedge\partial\bar \partial \omega^{n-2}=0.
$$
and this implies $\partial \phi=0$.
\end{proof}

A more general obstruction to the existence of astheno-K\"ahler metrics first 
noticed in \cite{lyz} (see also \cite[Corollary 2.2]{fgv}) is given by the Harvey-Lawson 
type criterion:

\begin{thm}
\label{hl-obs}
If a compact complex manifold $M$ admits a weakly positive and 
$\partial\bar\partial$-exact and non-vanishing $(2, 2)$-current, then it does 
not admit an astheno-K\"ahler metric.
\end{thm}
\proof Suppose $i\partial\bar\partial T$  is weakly positive, where we consider 
$T$ as a form with distribution coefficients. Then for an astheno-K\"ahler 
metric $\omega$ we have by integration by parts
$$
0 <\int_M i \partial\bar\partial T\wedge \omega^{n-2} =
\int_M T\wedge i \partial\bar\partial \omega^{n-2} = 0,
$$
which gives a contradiction.
\qed
\begin{rmk}
Note that for a holomorphic one-form $\alpha,$ the form 
$i(\partial \alpha \wedge \bar\partial \bar \alpha) = 
i\partial\bar\partial(\alpha\wedge \bar \alpha)$ 
is weakly positive, which leads to the obstruction of Jost and Yau in Lemma \ref{jy}. 
\end{rmk}
\begin{rmk}
Furthermore, we remark that the statement of Corollary \ref{hl-obs} cannot be 
reversed, since in general not every positive ($n-2, n-2)$-form arises as 
$(n-2)$-power of a positive $(1,1)$-form. 
\end{rmk}

%%%%%%%%%%%%%%%%%%%%%%%%%%%%%%%%%%%%%%%%%%%%%%%%%%%
%%%%%%%%%%%%%%%%%%%%%%%%%%%%%%%%%%%%%%%%%%%%%%%%%%%
%%%%%%%%%%%%%%%%%%%%%%%%%%%%%%%%%%%%%%%%%%%%%%%%%%%
%%%%%%%%%%%%%%%%%%%%%%%%%%%%%%%%%%%%%%%%%%%%%%%%%%%

 \section{Bott-Chern and Aeppli cohomologies} 
 \label{BCcoh}

Let $X$ be a compact complex manifold of dimension $n,$ and denote 
by  ${\mathcal A}^{p,q}(X)$ its space of smooth $(p,q)$-forms. The 
Bott-Chern cohomology groups are 
\begin{equation*}
H^{p,q}_{BC}(X,{\mathbb C})=
\frac{\{\alpha\in {\mathcal A}^{p,q}(X)\, \vert\, d\alpha=0\}}
{\{i\partial\bar\partial\beta\, \vert\,  \beta\in {\mathcal A}^{p-1,q-1}(X)\}},
\end{equation*}
while the Aeppli cohomology groups are
\begin{equation*}
H^{p,q}_A(X,{\mathbb C})=
\frac{\{\alpha\in {\mathcal A}^{p,q}(X)\, \vert\, i\partial\bar\partial \alpha=0\}}
{\{\partial\beta+\bar\partial\gamma\, \vert\, \beta\in {\mathcal A}^{p-1,q}(X),
\gamma\in {\mathcal A}^{p,q-1}(X)\}}.
\end{equation*} 

\smallskip

The Bott-Chern and Aeppli cohomologies satisfy a remarkable symmetry
property \cite{sch}: 
\begin{equation}
\label{symmetry}
H^{p,q}_{\#}(X)=\overline{H^{q,p}_{\#}(X)},
\end{equation}
where $\#\in \{BC, A\}.$ Furthermore, the groups 
$H^{p,q}_{BC}(X, {\mathbb C})$ and 
$H^{n-p,n-q}_A(X, {\mathbb C})$ are dual via the pairing
\begin{equation} 
H^{p,q}_{BC}(X, {\mathbb C})\times H_A^{n-p,n-q}
(X, {\mathbb C})\to {\mathbb C}, 
([\alpha ],\{\beta\})\to 
\int_X\alpha\wedge\beta.
\end{equation}

We will denote by $[\eta]$ the class of a $d$-closed $(p,q)$-form $\eta$ in 
$H^{p,q}_{BC}(X)$ and by $\{\zeta\}$ the class of a $i\partial\bar\partial$-closed 
$(p,q)$-form $\zeta$ in $H^{p,q}_A(X)$. Notice that the elements in $H^{0,1}_{BC}(X)$ 
are the $d$-closed $(0,1)$-forms on $X.$

%%%%%%%%%%%%%%%%%%%%%%%%%%%%%%%%%%%%%%%%%%
%%%%%%%%%%%%%%%%%%%%%%%%%%%%%%%%%%%%%%%%%%
%%%%%%%%%%%%%%%%%%%%%%%%%%%%%%%%%%%%%%%%%%

\subsection{A (very) weak K\"unneth formula}

Let $X$ and $Y$ two compact complex manifolds of dimensions $\dim X=n$ 
and $\dim Y=m.$ Let 
$$
p: X\times Y\to X\quad{\text{and}}\quad  q:X\times Y\to Y
$$ 
be two natural projections.

\begin{proof}[Proof of Theorem \ref{weak-kunneth}]

For the statement regarding the Bott-Chern cohomology, it suffices to show 
that the natural map 
$$
{\mathfrak s}_{BC}:H^{0,1}_{BC}(X)\oplus H^{0,1}_{BC}(Y)\to H^{0,1}_{BC}(X\times Y), 
\quad (\alpha,\beta)\mapsto p^*\alpha+q^*\beta
$$
is an isomorphism.

\smallskip

Let $\alpha\in  H^{0,1}_{BC}(X)$ and $\beta\in H^{0,1}_{BC}(Y)$ such 
that  $(\alpha,\beta)\in \ker ({\mathfrak s}^{0,1}_{BC}).$ That means 
$\alpha\in \mathcal A^{0,1}(X)$ and $\beta \in \mathcal A^{0,1}(Y)$ are 
$d$-closed forms in $X$ and $Y,$ respectively, and $p^*\alpha+q^*\beta=0$ in 
$\mathcal A^{0,1}(X\times Y)$. The latter automatically implies $\alpha=0$ 
and $\beta=0,$ and the injectivity of ${\mathfrak s}^{0,1}_{BC}$ is proven. 
To show its surjectivity,  let $\eta\in H^{0,1}_{BC}(X\times Y)$. We have
$\partial \eta=\bar \partial \eta=0,$ which imply $\bar\partial \bar \eta=0.$
Consider now the class $[\overline\eta]\in H_{\bar\partial}^{1,0}(X\times Y)$ in 
the Dolbeault cohomology of $X\times Y.$ From the K\"unneth formula for the 
Dolbeault cohomology \cite[page 105]{gh}, it follows that there exist 
$\alpha\in \mathcal A^{0,1}(X)$ and $\beta \in \mathcal A^{0,1}(Y)$
such that $\partial\alpha=0$, $\partial\beta=0$ and such that 
\begin{equation}
\label{BC-surjectivity}
\eta=p^*\alpha+q^*\beta.
\end{equation} 
Since $\eta$ is $d$-closed, it follows that 
$\bar\partial\eta=p^*\bar\partial\alpha+q^*\bar\partial\beta$ 
and if we consider local coordinates on $X$ and $Y$, the above equality implies 
that $\bar\partial\alpha=0$ and $\bar\partial\beta=0$.  Therefore $d\alpha=0$ and 
$d\beta=0,$ which together with (\ref{BC-surjectivity}) implies that the map 
${\mathfrak s}_{BC}$ is surjective, as well.

\smallskip

For the statement regarding the Aeppli cohomology of  $X\times Y,$ fix Hermitian metrics 
$\omega_X$  and $\omega_Y$ on $X$ and  $Y,$ respectively and consider the set 
$$
{\mathcal H}_A^{0,1}(X)=\{\alpha\in {\mathcal A}^{0,1}(X)\, \vert\, 
\partial\bar\partial\alpha=0, \bar\partial^*\alpha=0\},
$$ 
where $\bar\partial^*$ is the adjoint of $\bar\partial$ is considered with respect 
to $\omega_X$. It is known from \cite{sch} that this group is isomorphic to the 
Aeppli group $H^{0,1}_A(X).$  We also consider similarly defined sets 
${\mathcal H}_A^{0,1}(Y)$ and ${\mathcal H}_A^{0,1}(X\times Y),$ where we 
equip $X\times Y$ with the product Hermitian metric $p^*\omega_X+q^*\omega_Y.$ 
We define the map 
$$
{\mathfrak t}_A: {\mathcal H}_A^{0,1}(X)\oplus {\mathcal H}_A^{0,1}(Y)
\to {\mathcal A}^{0,1}(X\times Y), \quad (\alpha,\beta)\mapsto p^*\alpha+q^*\beta.
$$ 
This map is clearly injective. In order to prove the inequality 
$$
h_A^{0,1}(X)+h_A^{0,1}(Y)\leq h_A^{0,1}(X\times Y),
$$ 
it is enough to show that the range of ${\mathfrak t}_A$ is included in 
${\mathcal H}_A^{0,1}(X\times Y).$ That is equivalent to showing that if $\alpha$ 
and $\beta$ are harmonic, then $\eta:=p^*\alpha+q^*\beta$ is harmonic.\, i.e., 
$\eta$ satisfies $\partial \bar\partial\eta=\bar\partial^*\eta=0$ provided 
$\partial \bar\partial\alpha=\bar\partial^*\alpha=0$ and 
$\partial \bar\partial\beta=\bar\partial^*\beta=0.$ It is clear that $\partial\bar\partial\eta=0.$ 
and we will  show next that $\bar\partial^*\eta=0.$ Before we proceed, we recall 
a few standard results in complex geometry.

\smallskip

Let $\star$ be the Hodge operator acting on the space of $(0,1)$-forms on a compact 
complex manifold of complex dimension $p,$ equipped with a Hermitian metric with 
fundamental for $\omega.$ Then the adjoint operator $\bar\partial^*$ acting on $(0,1)$-forms 
is given by $\displaystyle \bar\partial^*=-\star\partial\star.$ Also, according to 
\cite[Proposition 6.29, p.150]{voisin}, for $(0,1)$-forms we have  
\begin{equation}
\star \gamma=\frac{i}{(p-1)!}\omega^{p-1}\wedge\gamma.
\end{equation}
Therefore, a $(0,1)$-form $\gamma$ satisfies $\bar\partial^*\gamma=0$ if and only if 
$\partial (\gamma \wedge \omega^{p-1})=0.$

\smallskip

We will prove now that $\bar\partial^* p^*\alpha=0.$  By the above consideration, 
it suffices to show that 
\begin{equation*}
\partial \left(p^*\alpha\wedge (p^*\omega_X+q^*\omega_Y)^{m+n-1}\right)=0,
\end{equation*}
which for degree reasons is equivalent to proving 
\begin{equation*}
\partial \left(p^*\alpha\wedge p^*\omega_X^{n-1}\wedge q^*\omega_Y^{m}\right)=0,
\end{equation*}

As $\alpha$ is harmonic, we know that $\partial(\alpha\wedge \omega_X^{n-1})=0,$ 
and compute:
\begin{align*}
\partial \left(p^*\alpha\wedge p^*\omega_X^{n-1}\wedge q^*\omega_Y^{m}\right)= 
&\, \partial( p^*\alpha\wedge p^*\omega_X^{n-1})\wedge q^*\omega_Y^m\\
=&\, \partial p^*(\alpha \wedge \omega_X^{n-1})\wedge q^*\omega_Y\\
= &\, p^*(\partial(\alpha \wedge \omega_X^{n-1}))\wedge q^*\omega_Y^m\\
= &\, 0.
\end{align*}

Hence $ \bar\partial^*p^*\alpha=0$ if $ \bar\partial^*\alpha=0,$ and a similar proof shows that 
$ \bar\partial^*q^*\beta=0$ if $ \bar\partial^*\beta=0,$ implying that $\bar\partial^*\eta=0.$ 
Therefore ${\mathcal H}_A^{0,1}(X\times Y)$ contains the range of $\mathfrak t_A$, and the 
inequality 
$$
h_A^{0,1}(X)+h_A^{0,1}(Y)\leq h_A^{0,1}(X\times Y)
$$ 
follows easily.
\end{proof}

%%%%%%%%%%%%%%%%%%%%%%%%%%%%%%%%%%%%%%%%%%
%%%%%%%%%%%%%%%%%%%%%%%%%%%%%%%%%%%%%%%%%%
%%%%%%%%%%%%%%%%%%%%%%%%%%%%%%%%%%%%%%%%%%

\subsection{The Bott-Chern cohomology of the blow-ups}
 
 We recall here without proof a recent result of Stelzig, computing the 
 Bott-Chern cohomology of the blow-ups.

 \begin{thm}[Stelzig \cite{stelzig-blow}]
 Let $Y$ be a compact complex manifold of complex dimension 
 $n\geq 2,\,  Z \subset Y$ be a closed complex submanifold of codimension 
 $c\geq 2,$ and $f:X\ra Y$ be the blow-up of $Y$ with center $Z.$ Then, 
 there exists an isomorphism
 $$
 H^{p,q}_{BC}(X) \simeq H^{p,q}_{BC}(Y) \oplus 
 \left(\bigoplus_{i=1}^{c-1} H^{p-i,q-i}_{BC}(Z)\right)
$$
for $p, q \leq n.$
\end{thm}

As an immediate consequence we notice the following:

\begin{cor}
\label{invariance} 
Let $Y$ be a compact complex manifold of complex dimension 
$n\geq 2,\,  Z \subset Y$ be a closed complex submanifold of codimension 
$c\geq 2,$ and $f:X\ra Y$ be the blow-up of $Y$ with center $Z.$ Then
\begin{itemize}
\item[ 1)] $h^{0,p}_{BC}(X)=h^{0,p}_{BC}(Y)$ for every $p\geq 0.$
\item[ 2)] $h^{0,1}_{A}(X)=h^{0,1}_{A}(Y).$  
\end{itemize}
\end{cor}

\begin{rmk}
As a consequence of the weak factorization theorem \cite{wft}, and the duality 
between the Bott-Chern and Aeppli cohomologies, we find that 
$h^{0,p}_{BC},\, p\geq 0$ and $h^{0,1}_{A}$ are bimeromorphic invariants.
\end{rmk}

%%%%%%%%%%%%%%%%%%%%%%%%%%%%%%%%%%%%%%%%%%
%%%%%%%%%%%%%%%%%%%%%%%%%%%%%%%%%%%%%%%%%%
%%%%%%%%%%%%%%%%%%%%%%%%%%%%%%%%%%%%%%%%%%

\subsection{The Bott-Chern cohomology of complex nilmanifolds.} 
Let $M=\Gamma\backslash G$ be a  compact nilmanifold equipped with a 
left-invariant integrable almost complex  structure $J,$ that is  $J$ comes from a
(left invariant) complex structure, also denoted by $J,$ on the Lie algebra 
$\mathfrak g$ of $G.$ According to Nomizu's Theorem \cite{no}, the de Rham 
cohomology of a compact nilmanifold can be computed by means of
the cohomology of the Lie algebra of the corresponding nilpotent Lie group, 
where the differential in the complex $\wedge^\bullet \mathfrak g^*$ is the 
Chevalley-Cartan differential. This result was refined  
to compute the Dolbeault and Bott-Chern cohomology of complex nilmanifolds 
with nilpotent complex structures in \cite{cfgu} and  \cite{angella}, respectively.
For convenience, we recall the 
main result of \cite{cfgu} which provides a useful tool 
for computing the cohomology of compact nilmanifolds 
with nilpotent complex structures.

\begin{thm}[Main Theorem, \cite{cfgu}]
\label{cfgu-nil}
Let $\Gamma\backslash G$ be a compact nilmanifold with a nilpotent complex 
structure, and let $\mathfrak g$ be the Lie algebra of $G.$ Then there is a 
quasi-isomorphism of complexes 
$$
\left(\wedge^{p,\bullet}\mathfrak g^*_\CC,\bar \partial\right)
\hookrightarrow 
\left (\wedge^{p,\bullet} (\Gamma\backslash G) ,\bar \partial\right).
$$
with respect to the operator $\bar \partial$ in the canonical decomposition 
$d = \partial +\bar \partial$ of the Chevalley-Eilenberg differential 
in $\wedge^\bullet (\mathfrak g^*_\CC).$
\end{thm}

As a corollary of Theorem \ref{cfgu-nil}, Angella proved  in \cite[Theorem 3.8]{angella} 
that  the inclusions of the corresponding subcomplexes yield isomorphisms 
$$
H^{p,q}_\#(\mathfrak g_\CC)\xhookrightarrow{\simeq} H^{p,q}_\#(M),
$$
where $\#\in\{BC, A\}.$ We will conclude this section by using Angella's result to 
prove Theorem \ref{torality}.

\begin{proof}[Proof of Theorem \ref{torality}] 
By the symmetry of Bott-Chern and Aeppli cohomologies, we can assume 
$h_{BC}^{1,0}(M)=h_A^{1,0}(M).$ In particular, we have 
\begin{align*}
 H^{1,0}_{BC}(M)\simeq &\, H^{1,0}_{BC}(\mathfrak g_\CC)
 =\{\alpha\in \mathfrak g^{1,0}\,|\, d\alpha=0\}\\
 H^{1,0}_{A}(M)\simeq &\, H^{1,0}_A(\mathfrak g_\CC)
 =\{\eta\in \mathfrak g^{1,0}\,|\, \ddbar \eta=0\}.
\end{align*}

Accordingly, it suffice to show that if 
$h_{BC}^{1,0}(\mathfrak g_\CC)=h_A^{1,0}(\mathfrak g_\CC),$ then 
${\mathfrak g}$ is abelian. Since $M=\Gamma\backslash G$ is 
equipped with a nilpotent complex structure, by Theorem \ref{equiv-nil}, 
there exists a basis $\{\alpha_i, \, 1\leq i\leq n \}$ of the $i$-eigenspace 
$\mathfrak g^{1,0}$ of the extension of $J$ to 
$\mathfrak g^*_\CC=\mathfrak g^*\otimes_\RR \CC$ such that the 
structure equations of $G$ are of the form 
(\ref{structure}). 

Notice first from (\ref{structure}) that we have $d\alpha_1=0.$ Suppose now 
$d\alpha_{\ell}=0$ for every $1\leq {\ell}<s.$ Since $d\alpha_{\ell}=0,$ we have 
$\partial\alpha_{\ell}=\bar\partial\alpha_{\ell}=\partial \bar\alpha_{\ell}
=\bar\partial\bar \alpha_{\ell}=0,\, 1\leq {\ell}< s.$ 
From the structure equations (\ref{structure}) we can immediately see that 
$$
\partial \alpha_s=\sum_{j<k<s} A^s_{j,k}\, \alpha_j\wedge \alpha_k
\quad{\text{and}}\quad \bar \partial \alpha_s=
\sum_{j,k<i} B^s_{j,k}\, \alpha_j\wedge \bar \alpha_k.
$$
Then $\ddbar \alpha_s=0,$ which means 
$\alpha_s\in H^{1,0}_A(\mathfrak g_\CC)=H^{1,0}_{BC}(\mathfrak g_\CC).$ 
Therefore $d\alpha_s=0.$ By induction, it follows that $d\alpha_i=0,$ for every 
$i=1,\dots,n,$ which is equivalent to $\mathfrak g$ being abelian.
\end{proof}

%%%%%%%%%%%%%%%%%%%%%%%%%%%%%%%%%%%%%%%%%%%%%%%%%%%
%%%%%%%%%%%%%%%%%%%%%%%%%%%%%%%%%%%%%%%%%%%%%%%%%%%
%%%%%%%%%%%%%%%%%%%%%%%%%%%%%%%%%%%%%%%%%%%%%%%%%%%
%%%%%%%%%%%%%%%%%%%%%%%%%%%%%%%%%%%%%%%%%%%%%%%%%%%

 \section{A new obstruction}
 \label{g-arg}
  
Let $(M,\omega)$ be compact an astheno-K\"ahler manifold of complex dimension 
$n,$ and let $f$ be a ${\mathcal C}^{\infty}$ real function on $M$ such that 
$e^{(n-1)f}\omega^{n-1}$ is $i\partial\bar\partial$-closed. The existence of $f$ follows 
from Gauduchon's Theorem \cite{gauduchon}. We define now a linear map  
$L:H^{0,1}_A(M)\to {\mathbb C}$ as follows: 
\begin{equation}
\label{degree}
L(\{\alpha\})=\int_M\partial\alpha\wedge e^{(n-1)f}\omega^{n-1}.
\end{equation}
Note that $L$ is well-defined since $e^{(n-1)f}\omega^{n-1}$ is $\partial\bar\partial$-closed.

\smallskip

\begin{proof}[Proof of Theorem \ref{c-ineq}]
The natural morphism induced by the identity
$$
{\mathfrak i}_{0,1}:H^{0,1}_{BC}(M)\to H^{0,1}_A(M)
$$ 
is always injective, and $L({\iota}_{0,1}([\phi]))=0$ for all $[\phi]\in H^{0,1}_{BC}(M).$ 
Therefore we have a complex
\begin{equation}
\label{L-sequence}
0\to H^{0,1}_{BC}(M)\xrightarrow{{\mathfrak i}_{0,1}} H^{0,1}_A(M)
\xrightarrow{L} {\mathbb C}.
\end{equation}
To prove Theorem \ref{c-ineq}, it suffices to prove that the sequence (\ref{L-sequence}) 
is exact, that is to show that  $\ker(L)\subseteq H^{0,1}_{BC}(M).$ 

\smallskip

Let $\{\alpha\}\in \ker(L),$ i.e.,   
$
\displaystyle \int_M\partial\alpha \wedge e^{(n-1)f}\omega^{n-1}=0.
$
Consider the elliptic differential operator 
$$
P:{\mathcal C}^{\infty}(M)\to {\mathcal C}^{\infty}(M), 
\quad P(h)=\Lambda (\partial\bar\partial h)
$$ 
where $\Lambda$ is the adjoint of 
$$
{\mathcal C}^{\infty}(M)\ni g\to ge^f\omega
\in{\mathcal A}^{1,1}(M)
$$ 
with respect to the metrics induced by $e^f\omega.$ It is easy to see that the formal 
adjoint of $P$ is 
$$
P^*(h)=\bar\partial^*\partial^*(he^f\omega).
$$ 
Another way of writing $P^*$ is 
$$
P^*(h)=-\star\partial\bar\partial \left(h\frac{e^{(n-1)f}\omega^{n-1}}{(n-1)!}\right)
$$
and it is well known \cite{gauduchon} that, since 
$
\displaystyle \partial\bar\partial\left(e^{(n-1)f}\omega^{n-1}\right)=0,
$
the kernel of $P^*$ is $1$-dimensional. Namely, it consists of the constant 
functions on $M$. 

Since $P$ is elliptic, we have the following orthogonal decomposition:
$$
{\mathcal C}^{\infty}(M)=P({\mathcal C}^{\infty}(M))\oplus \ker(P^*)
$$
Therefore $\Lambda(\partial\alpha)\in P({\mathcal C}^{\infty}(M)$ if and only if 
$\Lambda(\partial\alpha)$ is orthogonal to $\ker(P^*)$, that is, if and only if 
$((\Lambda(\partial\alpha), 1))=0.$ But 
$$
((\Lambda(\partial\alpha), 1))=((\partial\alpha, e^f\omega))=
\int_M\partial\alpha\wedge\overline{\star e^f\omega}=
\int_M\partial\alpha\wedge\frac{e^{(n-1)f}\omega^{n-1}}{(n-1)!}=0
$$ 
since $\{\alpha\}$ is in the kernel of $L$ and 
$$
\star e^f\omega= \frac{e^{(n-1)f}\omega^{n-1}}{(n-1)!}
$$ 
where the Hodge star operator is with respect to $e^f\omega$. Therefore 
$\Lambda (\partial\alpha)$ is in the range of $P,$ which means that $\{\alpha\}$ 
admits a representative, also denoted by $\alpha$, such that $\Lambda(\partial\alpha)=0$. 
This means that $\partial\alpha$ is primitive with respect to $e^f\omega$. It is easy to see 
that $\partial\alpha$ is also primitive with respect to $\omega$, and therefore 
$$
\star\partial\alpha=-\frac{1}{(n-2)!}\partial\alpha\wedge\omega^{n-2}
$$ 
\cite[Proposition 6.29, p.150]{voisin},  where the Hodge star operator here is considered 
with respect to $\omega.$ Therefore the $L^2$ norm of $\partial\alpha$ with respect 
to $\omega$ is 
$$
\vert\vert\partial\alpha\vert\vert^2=
\int_M\partial\alpha\wedge\overline{\star\partial\alpha}=
-\frac{1}{(n-2)!}\int_M\partial\alpha\wedge
\bar\partial\bar\alpha\wedge\omega^{n-2}
$$ 
Since $\bar\partial\alpha$ is of type $(0,2)$,  it is also primitive with respect to $\omega,$ 
and so 
$$
\star\bar\partial\alpha=\frac{1}{(n-2)!}\bar\partial\alpha\wedge \omega^{n-2}
$$ 
\cite[Proposition 6.29, p.150]{voisin} and the $L^2$ norm of $\bar\partial\alpha$ with 
respect to $\omega$ is 
$$
\vert\vert\bar\partial\alpha\vert\vert^2=
\frac{1}{(n-2)!}\int_M\bar\partial\alpha\wedge
\partial\bar\alpha\wedge\omega^{n-2}.
$$

As the metric $\omega$ is astheno-K\"ahler, $\partial\bar\partial\omega^{n-2}=0,$ and it 
follows that 
$$
0=\int_M\partial\bar\partial(\alpha\wedge\bar\alpha)\wedge\omega^{n-2}=
\int_M(\bar\partial\alpha\wedge\partial\bar\alpha-
\partial\alpha\wedge\bar\partial\bar\alpha)\wedge \omega^{n-2}=
(n-2)!(\vert\vert\bar\partial\alpha\vert\vert^2+\vert\vert\partial\alpha\vert\vert^2)
$$
therefore $\partial\alpha=\bar\partial\alpha=0$, which means that 
$\alpha$ is in $H^{0,1}_{BC}(M)$.

Since the sequence 
$$
0\to H^{0,1}_{BC}(M)\to H^{0,1}_A(M)\to {\mathbb C}
$$ 
is exact, Theorem \ref{c-ineq} follows.
\end{proof}
\begin{cor}
Every SKT compact complex manifold of dimension three satisfies the inequalities (\ref{ineq}).
\end{cor}

\begin{rmk}
According to the Jost-Yau's Lemma \ref{jy}, on an astheno-K\"ahler manifold, every 
$\bar\partial$-closed $1$-form (i.e., a holomorphic $1$-form) is $d$-closed. It follows 
that the identity map on $(1,0)$-forms induces a morphism 
$$
H^{1,0}_{\bar\partial}(M)\to H^{1,0}_{BC}(M)
$$
inverse to the map 
$$
H^{1,0}_{BC}(M)\to H^{1,0}_{\bar\partial}(M),
$$ 
also induced by the identity.  Hence this map is an isomorphism, and so 
$h^{1,0}_{BC}=h^{1,0}_{\bar\partial}.$ Therefore, on an astheno-K\"ahler manifold 
we have 
$$
h^{1,0}_{A}-h^{1,0}_{\bar\partial}\in\{0,1\}.
$$
\end{rmk}
\begin{rmk}
The cohomology group $H^{0,1}_{\bar\partial}(M)$ lies between 
$H^{0,1}_A(M)$ and $H^{0,1}_{BC}(M)$, i.e., the natural morphisms 
$H^{0,1}_{\bar\partial}(M)\to H^{0,1}_A(M)$ and 
$H^{0,1}_{BC}(M)\to H^{0,1}_{\bar\partial}(M)$ are injective, so it follows that, on an 
astheno-K\"ahler manifold, at least one of the morphisms 
$$
H^{0,1}_{\bar\partial}(M)\to H^{0,1}_A(M)\qquad{\text{and}}\qquad  
H^{0,1}_{BC}(M)\to H^{0,1}_{\bar\partial}(M)
$$ 
is an isomorphism.
\end{rmk}
\begin{rmk}
\label{surfaces} 
Theorem \ref{c-ineq} generalizes in arbitrary dimension classical results on compact 
complex surfaces. Any compact complex surface is trivially astheno-K\"ahler, and if 
the surface also admits a K\"ahler metric, then $h^{0,1}_A=h^{0,1}_{BC}$ by the 
$\ddbar$-lemma. If the surface is non-K\"ahler, one has $h^{0,1}_A=h^{0,1}_{BC}+1.$ 
Indeed, from \cite[Theorem 3]{kodaira}, it is known that, on a non-K\"ahler surface, 
$h^{0,1}_{\bar\partial}=h^{1,0}_{\bar\partial}+1$. On the other hand, on compact complex 
surfaces, $h^{0,1}_{BC}=h^{1,0}_{BC}=h^{1,0}_{\bar\partial}$ and 
$h^{0,1}_{\bar\partial}=h^{0,1}_A$.
\end{rmk}

\begin{rmk}
From the above proof, it follows that on an astheno-K\"ahler manifold one has 
$h^{0,1}_{BC}=h^{0,1}_A$ if and only if $\partial( e^{(n-1)f}\omega^{n-1})$ is 
$\partial\bar\partial$-exact. Indeed, if $\partial( e^{(n-1)f}\omega^{n-1})$ 
is $\partial\bar\partial$-exact, the application $L$ defined above is zero, hence 
$h^{0,1}_{BC}=h^{0,1}_A$. Conversely, if $h^{0,1}_{BC}=h^{0,1}_A,$ the natural 
morphism 
$$
H^{n-1,n-1}_{BC}(M,{\mathbb C})\to H^{n-1,n-1}_A(M,{\mathbb C})
$$ 
is onto (see Proposition \ref{maps} below), implying that $\partial( e^{(n-1)f}\omega^{n-1})$ 
is $\partial\bar\partial$-exact since $e^{(n-1)f}\omega^{n-1}$ is $\partial\bar\partial$-closed.
\end{rmk}

\begin{prop}
\label{maps}
Let $M$ be a compact complex astheno-K\"ahler $n$-fold, and let 
$$
\frak i_{p,q}:H^{p,q}_{BC}(M)\to H^{p,q}_A(M)
$$ 
be the map induced by the identity. Then $\frak i_{n-1,n-1}$ is surjective if and only if 
$\frak i_{0,1}$ is injective.
\end{prop}
\begin{proof}
Before proceeding with the proof, notice, by duality, that $\frak i_{n-1,n-1}$ is surjective 
if and only if $\frak i_{1,1}$ is injective. That means $\frak i_{n-1,n-1}$ is surjective 
if and only if  for every $d$-closed $(1,1)$-form  $\phi$ such that 
$\phi=\partial \alpha+\bar\partial \beta,$ where $\alpha$ and $\beta$ are  smooth forms 
of type $(0,1)$ and $(1,0),$ respectively, there exists a smooth function 
$f:M\ra \CC$ such that $\phi=\partial\bar \partial f.$

\smallskip

Suppose the map  $\frak i_{0,1}$ is onto. That means for every $\partial \bar\partial$-closed 
$(0,1)$-form $\alpha,$ there exists a smooth function $g:M\ra \CC$ such that 
$\alpha +\bar \partial g$ is $d$-closed, i.e., $\bar \partial\alpha=0,$ and 
$\partial \alpha=-\partial \bar\partial g.$  By the symmetry of the Bott-Chern cohomology, 
it also follows that for every $\partial \bar\partial$-closed $(1,0)$-form $\beta,$ there exists 
a smooth function $h:M\ra \CC$ such that $\partial\beta=0,$ and 
$\bar \partial \beta=-\partial \bar\partial h.$ Let now  $\phi$  be a $d$-closed $(1,1)$-form 
such that $\phi=\partial \alpha+\bar\partial \beta,$ where $\alpha$ and $\beta$ are 
$(0,1)$ and $(1,0)$ smooth forms, respectively. Since $\phi$ is $d$-closed, then  
$\partial \bar\partial \alpha=\partial \bar\partial \beta =0.$ It follows that 
$\phi=-\partial \bar\partial (g+h).$

\smallskip

Conversely, let $\alpha$ be a  $(0,1)$-form such that $\partial \bar\partial \alpha=0.$ 
Let $\phi=\partial \alpha.$ Then $\phi$ is a real $d$-closed $(1,1)$-form in the kernel 
of $\frak i_{1,1}.$ Since $\frak i_{1,1}$ is by assumption injective, there exists a smooth 
function $f:M\ra \CC$ such that $\phi=\partial\bar\partial f,$ which means 
$\partial\alpha=\partial\bar\partial f.$ Therefore $\bar\partial (\bar\alpha-\partial\bar f)=0,$ 
and from the Jost-Yau result, it follows that  $\alpha-\bar\partial f$ is $d$-closed. That implies 
$\bar\partial \alpha=0$ which concludes the proof of the proposition.
\end{proof}

%%%%%%%%%%%%%%%%%%%%%%%%%%%%%%%%%%%%%%%%%%%%%%%%%%%%%%
%%%%%%%%%%%%%%%%%%%%%%%%%%%%%%%%%%%%%%%%%%%%%%%%%%%%%%
%%%%%%%%%%%%%%%%%%%%%%%%%%%%%%%%%%%%%%%%%%%%%%%%%%%%%%
%%%%%%%%%%%%%%%%%%%%%%%%%%%%%%%%%%%%%%%%%%%%%%%%%%%%%%

\section{Applications} 
\label{applications}

We proceed by giving two examples of solvmanifolds, where Theorem \ref{c-ineq} prohibits 
the existence of astheno-K\"ahler metrics.

\subsection{The complex parallelizable Nakamura manifold.}

An example where our cohomological obstruction in Theorem \ref{c-ineq} applies directly 
is the complex parallelizable Nakamura manifold. We will follow closely the description in 
\cite{ak}.
 
 \smallskip
 
Let $G $ be semidirect product $\CC \ltimes_{\phi} \CC^2$ defined by 
$$
\phi (z) = \left(\begin{array}{cc}e^z & 0 \\0 & e^{-z}\end{array}\right).
$$
Then there exist $a+bi, c+di\in\CC$ such that $\ZZ(a+bi)+\ZZ(c+di)$is a lattice in $\CC$ 
and $\phi (a+bi)$ and $\phi(c+di)$ are conjugate to elements of $SL(4;\ZZ),$ where we 
regard $SL(2;\CC) \subset SL(4;\RR)$. Hence we have a lattice 
$\Gamma:=\left (\ZZ(a+bi)+\ZZ(c+di)\right) \ltimes_{\phi} \Gamma''$ 
of $G$ such that $\Gamma''$ is a lattice of $\CC^2.$ The quotient manifold 
$\Gamma \backslash G$ is the compact complex parallelizable manifold constructed 
by  Nakamura \cite[\S2]{nak}. 
 
The Bott-Chern cohomology of the Nakamura manifold was computed by Angella and 
Kasuya \cite{ak}. They found that $h^{0,1}_{BC}=1,$  and if $b, d\in \ZZ\pi,$ then 
$h^{0,1}_{A}=5,$ while  $h^{0,1}_{A}=3$ if $b\notin\ZZ\pi$ or  $d\notin\ZZ\pi.$ 
As a direct consequence of Theorem \ref{ineq}, we have:
\begin{cor}
The complex parallelizable Nakamura manifold does not admit astheno-K\"ahler metrics.
\end{cor}

\begin{rmk}
Biswas proved  in \cite{biswas} that any compact complex parallelizable manifold admitting 
an astheno-K\"ahler metric is a compact complex torus. While Biswas' proof relies on the 
Jost and Yau's obstruction in Lemma \ref{jy}, in the case of the complex parallelizable 
Nakamura manifold, our proof relies on the new obstruction in Theorem \ref{c-ineq}.
\end{rmk}

%%%%%%%%%%%%%%%%%%%%%%%%%%%%%%%%%%%%%%%%%%
%%%%%%%%%%%%%%%%%%%%%%%%%%%%%%%%%%%%%%%%%%
%%%%%%%%%%%%%%%%%%%%%%%%%%%%%%%%%%%%%%%%%%

 \subsection{The Oeljeklaus-Toma manifolds}
 
The Oeljeklaus-Toma manifolds, introduced in \cite{ot} are interesting examples of compact, 
complex, non-K\"ahler manifolds, generalizing the Inoue surfaces \cite{inoue}.

 \medskip
 
Let $s$ and $t$ be two positive integers and consider $K\simeq \QQ[X]/(f)$ an algebraic 
number field, where $f\in\QQ[X]$ is a monic irreducible polynomial of degree $n=[K:\QQ]$ 
with $s$ real roots and $2t$ complex roots (\cite[Remark 1.1]{ot}). The field $K$ admits $s$ 
real embeddings and $2t$ complex embeddings:
$$
\sigma_1,\dots,\sigma_s:K\ra \RR 
$$
$$
\sigma_{s+1},\dots,\sigma_{s+2t}:K\ra \CC, 
{\text{where}}\, \sigma_{s+t+j}=\bar \sigma_{s+j},  j=1,\dots,t.
$$
The ring $\OO_K$ of algebraic integers of $K$ is a finitely-generated free Abelian group of 
rank $n.$ By the Dirichlet unit theorem, the multiplicative group $\OO_K^*$ of units of 
$\OO_K$ is a finitely-generated free Abelian group of rank $s + t -1.$ Furthermore, let
$$
\OO^{*, +}_K=\{a\in \OO_K^*\, |\, \sigma_i(a)>0, 
\, {\text{for all}}\,  i=1\dots, s\}.
$$ 
be the subgroup totally positive units. It is a finite index subgroup of $\OO_K^*.$

Denote now by $\mathbb H := \{z\in \CC\, |
\,  \mathfrak{Im}(z) > 0\}$ 
the upper complex half-plane. There exists an action 
$$
\OO_K\times \OO_K^{*,+} 
\circlearrowleft \HH^s\times \CC^t,
$$
induced by the translation 
$T: \OO_K \circlearrowleft \HH^s\times \CC^t$ given by 
$$
T_a(w_1,\dots,w_s, z_{s+1},\dots,z_{s+t})
=(w_1+\sigma_1(a),\dots,\dots,z_{s+t}+\sigma_{s+t}(a)),
$$
and by the multiplication $R: \OO_K^{*,+} \circlearrowleft \HH^s\times \CC^t$ given by 
$$
R_u(w_1,\dots, w_s, z_{s+1},\dots,z_{s+t})
=(w_1\cdot \sigma_1(u),\dots,\dots,z_{s+t}\cdot \sigma_{s+t}(u)).
$$

According to \cite[page 162]{ot},  one can always choose a rank $s$ subgroup 
$U\subset \OO^{*,+}_K$ such that the the induced action 
$\OO_K\times U\circlearrowleft \HH^s\times \CC^t$ is fixed-point-free, properly 
discontinuous and co-compact. One defines the Oeljeklaus-Toma manifold of type $(s, t)$
associated to the algebraic number field $K$ and to the admissible subgroup $U$ of 
$\OO_K^{*,+}$ as 
$$
X(K,U)= (\mathbb H\times \CC^t)/\OO_K\times U.
$$
 
A result of Kasuya \cite{kasuya} shows that the Oeljeklaus-Toma manifolds are 
solvmanifolds, a result with several consequences. In particular, one can use it to 
obtain information about the Dolbeault and Bott-Chern cohomology of such manifolds 
(see \cite{ados} and the references therein). In the following lemma, we compute the 
cohomology groups $h^{0,1}_{\#},$ for $\#\in \{BC, A\}.$  The authors would like to 
thank A. Otiman for kindly communicating to us the proof of this result.
\begin{lemma}
\label{otbc}
Let $X$ be an Oeljeklaus-Toma manifold of type $(s,t).$ Then 
$$
h^{0,1}_{BC}(X)=0\quad{\text{and}}\quad h^{0,1}_A(X)=s.
$$
\end{lemma}
\begin{proof}
By \cite[Corollary 11]{ados}, we obtain
$$
h^{0,1}_{BC}(X)=h^{1,0}_{\bar\partial}(X)=0,
$$
where the last equality was proved in \cite[Proposition 2.4]{ot}. 
 
By duality, from the same \cite[Corollary 11]{ados}, we also have 
$$
h^{0,1}_{A}(X)=h^{n,n-1}_{BC}(X)=h^{n,n-1}_{\bar\partial}(X)+h^{n-2,n}_{\bar\partial}(X).
$$
We will show that $h^{n,n-1}_{\bar\partial}(X)=s$ and $h^{n-2,n}_{\bar\partial}(X)=0.$

Let  
 $$
 \rho_{p,m}:=\#\{I\subseteq \{1,,\dots,s+t\}, J\subseteq  \{s+1,\dots,s+t\}\,
 \vert \, |I|=p, \|J|=m, \sigma_I\bar\sigma_J=1\},
 $$
 where for a subset $K\subseteq\{1,\dots,s+t\}, \, \sigma_K: =\prod_{i\in K}\sigma_i.$

By \cite[(18)]{ados} we have 
\begin{equation}
\label{d-ados}
h^{n,n-1}_{\bar\partial}(X)=\sum_{l=0}^{n-1}\binom{s}{l}\cdot \rho_{n,n-1-l}.
\end{equation}
Notice now that $n=s+t$ and $\rho_{i,j}=0$ if $j>t$ while $\rho_{n,t-1}=0,$ and so 
$$
h^{n,n-1}_{\bar\partial}(X)=\binom{s}{s-1}\cdot \rho_{n,t}=s,
$$
as $\rho_{n,t}=1.$ Similarly, from (\ref{d-ados}), we find 
$$
h^{n-2,n}_{\bar\partial}(X)=\binom{s}{s}\cdot \rho_{n-2,t}=\rho_{2,0}.
$$
However, by \cite[page 2419]{apv}, we have $\rho_{2,0}=0,$ and so 
$$
h^{0,1}_{A}(X)=h^{n,n-1}_{BC}(X)=s.
$$
\end{proof} 
  
As a consequence of Theorem \ref{c-ineq} and  Lemma \ref{otbc}, for $s\geq 2$ we obtain a 
new proof of a recent result of Angella, Dubickas, Otiman and Stelzig \cite[Corollary 5]{ados}.

\begin{cor}
The Oeljeklaus-Toma manifolds of type $(s,t)$ with  $s\geq 2$ do not admit astheno-K\"ahler 
metrics.
\end{cor}
 
\begin{rmk}
The Oeljeklaus-Toma manifolds with $s=1$ do not  admit astheno-K\"ahler, as well  
\cite[Corollary 5]{ados}.
\end{rmk}

%%%%%%%%%%%%%%%%%%%%%%%%%%%%%%%%%%%%%%%%%%
%%%%%%%%%%%%%%%%%%%%%%%%%%%%%%%%%%%%%%%%%%
%%%%%%%%%%%%%%%%%%%%%%%%%%%%%%%%%%%%%%%%%%

\subsection{Cartesian products} 
\label{c-prod}

A direct application of our obstruction and of the weak K\"unneth formula 
is the following observation:
\begin{prop}
Let $X$ and $Y$ be two astheno-K\"ahler manifolds saturating the upper bound 
in (\ref{ineq}). Then $X\times Y$ does not admit astheno-K\"ahler metrics.
\end{prop}

\proof 
As a consequence of Theorem \ref{weak-kunneth} we find 
\begin{align*}
h^{0,1}_A(X\times Y)\geq &\, h^{0,1}_A(X)+h^{0,1}_A(Y)\\
= &\, h^{0,1}_{BC}(X)+h^{0,1}_{BC}(Y)+2\\
=& \, h^{0,1}_{BC}(X\times Y)+2.
\end{align*}
According to Theorem \ref{c-ineq}, this prohibits the existence of an astheno-K\"ahler 
metric on $X\times Y.$
\qed

\begin{cor}
\label{obs-prod-surf}
Let $S_1$ and $S_2$ be two non-K\"ahler surfaces. Then $S_1\times S_2$ does not 
admit astheno-K\"ahler metrics.
\end{cor}

However, more is true. In \cite{lyz}, Li, Yau and Zheng noticed that a source 
of astheno-K\"ahler manifolds is provided by the products of the form $C\times S,$ 
where $C$ is a compact Riemann surface and $S$ is a compact, complex, 
non-K\"ahler surface. We generalize here their method of constructing 
astheno-K\"ahler manifolds. The examples produced will shed a light on the 
sharpness of the bounds in (\ref{ineq}).
 
\begin{prop}
 \label{astheno-simple}
Let  $S$ be a compact complex surface, $X$ a compact K\"ahler manifold of 
complex dimension $n.$ If $Y$ denotes the blow-up of $X\times S$ along a smooth 
submanifold, then $Y$ is an astheno-K\"ahler manifold and 
\begin{itemize}
\item[ 1)] $h^{0,1}_A(Y)= h^{0,1}_{BC}(Y)+1$
if and only if $S$ is non-K\"ahler.
\item[ 2)] $h^{0,1}_A(Y)= h^{0,1}_{BC}(Y)$
if and only if $S$ is K\"ahler.
\end{itemize}
\end{prop}
\begin{proof} 
By Theorem \ref{fino-tomassini}, to prove that $Y$ carries an astheno-K\"ahler 
metric it suffices to show that the product manifolds $S\times X$ carries an 
astheno-K\"ahler metric satisfying the condition (\ref{ft-condition}). Let $\omega_X$ 
be a  K\"ahler metric on $X$ and $\omega_S$ be a Gauduchon metric on $S.$ 
Notice that such a metric $\omega_S$ exists on arbitrary compact complex manifolds 
\cite{gauduchon}. Consider now the Hermitian metric  $\omega=\omega_X+\omega_S$ 
on $X\times S,$ where by abusing the notation, we omit the pull-back maps from the 
two factors. We compute
 \begin{equation}
 \label{basis}
 i\ddbar \omega^n= i\ddbar(\omega_X+\omega_S)^n
 = i\ddbar\omega_X^n+ni\ddbar (\omega_X^{n-1}\wedge\omega_S)+
 \binom{n}{2}i\ddbar(\omega_X^{n-2}\wedge\omega_S^2).
 \end{equation}
 Since $\omega_X$ is K\"ahler, we have 
 $\partial \omega_X=\bar \partial \omega_X=0,$ 
 which implies 
 $$i\ddbar (\omega_X^{n-1}\wedge\omega_S)=-\omega_X^{n-1}\wedge 
 (i\ddbar \omega_S)=0,
 $$ 
since $\omega_S$ is Gauduchon. Similarly,
$$
\ddbar(\omega_X^{n-2}\wedge\omega_S^2) = -\omega_X^{n-2}\wedge  
(i\ddbar \omega_S^2)=0,
$$ 
since $i\ddbar \omega_S^2=0$ is automatically satisfied on surfaces. Finally, the term 
$i\ddbar\omega_X^n$ vanishes for degree reasons. In conclusion, from (\ref{basis}) we 
see that $ i\ddbar \omega^n=0,$ which means $\omega$ is an astheno-K\"ahler metric. 
Also, the same arguments can be easily adapted to show $\ddbar \omega= \ddbar \omega^2=0,$ 
which shows that the metric $\omega$ satisfies the condition (\ref{ft-condition}). 

\smallskip
 
To prove claims 1) and 2), from  Corollary \ref{invariance} and Theorem 
\ref{weak-kunneth},  we find:
\begin{align}
\label{master}
  h^{0,1}_{BC}(Y)= & \, h^{0,1}_{BC}(X\times S)=h^{0,1}_{BC}(X)
  + h^{0,1}_{BC}(S) = q+ h^{0,1}_{BC}(S) \\ \notag
    h^{0,1}_{A}(Y)= & \, h^{0,1}_{A}(X\times S)\geq h^{0,1}_{A}(X)
  + h^{0,1}_{A}(S) = q+ h^{0,1}_{A}(S)
\end{align}
where $q=h^{0,1}_{\bar\partial}(X)=h^{0,1}_{BC}(X)=h^{0,1}_{A}(X)$ is the 
irregularity of the K\"ahler manifold $X.$ Therefore, if $h^{0,1}_A(Y)= h^{0,1}_{BC}(Y)$ 
then $h^{0,1}_{A}(S)=h^{0,1}_{BC}(S).$ As in Remark \ref{surfaces}, 
the latter is equivalent to $S$ being a K\"ahler surface.  Notice now that 
if $S$ is K\"ahler, then $X\times S$ is K\"ahler, and so $Y$ is K\"ahler. 
That implies $h^{0,1}_A(Y)= h^{0,1}_{BC}(Y),$ proving not only claim 1), but also 
one direction of claim 2). Indeed, if $h^{0,1}_A(Y)= h^{0,1}_{BC}(Y)+1$ then 
$S$ is necessarily non-K\"ahler. Conversely, if $h^{0,1}_{A}(S)=h^{0,1}_{BC}(S)+1,$ 
by (\ref{master}) $h^{0,1}_A(Y)\geq h^{0,1}_{BC}(Y)+1.$ However, since $Y$ carries 
an astheno-K\"ahler metric, by Theorem \ref{c-ineq}, we must have 
$h^{0,1}_A(Y)=h^{0,1}_{BC}(Y)+1.$ 
\end{proof}
 \begin{proof}[Proof of Theorem \ref{cart-surfaces}] 
The result is a direct consequence of Proposition \ref{astheno-simple} 
and Corollary \ref{obs-prod-surf}.
\end{proof}
\begin{rmk}
\label{non-jy}
We notice here that the Jost-Yau obstruction criterion does not provide any information 
in proving that a Cartesian product of two non-K\"ahler surfaces does not carry 
astheno-K\"ahler metrics, emphasizing the role played by Theorem \ref{c-ineq} 
in the proof of Theorem \ref{cart-surfaces}. Indeed, let $S_1$ and $S_2$ be 
two non-K\"ahler surfaces and $\phi$ a holomorphic $1$-form on $S_1\times S_2.$  
Since $\partial \alpha=0,$ then $\alpha $ is an element of the Dolbeault 
cohomology group $H^{1,0}_{\bar\partial}(X\times Y).$ Using the K\"unneth 
formula \cite[page 105]{gh}, there exist holomorphic $1$-forms $\alpha_1$ 
and $\alpha_2$ on $S_1$ and  $S_2,$ respectively, such that 
$\phi=\alpha_1+\alpha_2.$ However, on surfaces, every holomorphic $1$-form 
is closed as it can be seen from the Jos-Yau criterion, or its proof. Therefore 
$d\alpha_1=0$ and $d\alpha_2=0,$ which imply 
$d\phi=d(\alpha_1+\alpha_2)=d\alpha_1+d\alpha_2=0.$  
\end{rmk}

\begin{rmk}
It is interesting to notice that the product of either two Kodaira surfaces, 
two Inoue surfaces, ora Kodaira surface and a Inoue surface, while it does 
not carry any astheno-K\"ahler metric, it does  admit SKT metrics 
\cite[Theorem 7.5]{st}.
\end{rmk}

%%%%%%%%%%%%%%%%%%%%%%%%%%%%%%%%%%%%%%%%%%%%%%%%%%%%%%
%%%%%%%%%%%%%%%%%%%%%%%%%%%%%%%%%%%%%%%%%%%%%%%%%%%%%%
%%%%%%%%%%%%%%%%%%%%%%%%%%%%%%%%%%%%%%%%%%%%%%%%%%%%%%
%%%%%%%%%%%%%%%%%%%%%%%%%%%%%%%%%%%%%%%%%%%%%%%%%%%%%%

 \section{Saturating the inequality (\ref{ineq})}
 \label{examples}

The results presented earlier show that the upper bound in (\ref{ineq}) is sharp: 
the Cartesian product of a compact, complex, non-K\"ahler surface and a compact 
K\"ahler manifold admits astheno-K\"ahler metrics and saturates the upper bound 
in (\ref{ineq}). We will exhibit next another class of non-K\"ahler manifolds which do 
not carry  astheno-K\"ahler metrics, the class of Vaisman manifolds. The diagonal Hopf 
manifolds are examples of  Vaisman manifolds.

%%%%%%%%%%%%%%%%%%%%%%%%%%%%%%%%%%%%%%%%%%
%%%%%%%%%%%%%%%%%%%%%%%%%%%%%%%%%%%%%%%%%%
%%%%%%%%%%%%%%%%%%%%%%%%%%%%%%%%%%%%%%%%%%

\subsection{Vaisman manifolds} 

Locally conformally K\"ahler (LCK) manifolds and, in particular, the Vaisman manifolds 
were introduced in the 70's by I. Vaisman, and have been given a lot of attention in the 
recent years. We will recall next some of the relevant results.

\smallskip

Let $M$ be a complex manifold. A Hermitian metric on $M$ is a locally conformally 
K\"ahler (LCK) metric if its fundamental form $\omega$ satisfies the condition 
$d\omega=\theta \wedge \omega$ for some non-zero $1$-form $\theta.$ 
The $1$-form $\theta$ is called the Lee form, and the pair $(M,\omega)$ is called a 
LCK manifold. If $\dim M\, \geq 3,$ the Lee form is closed. We will decompose the  
$1$-form $\theta$  as $\theta = \alpha + \bar \alpha$ 
where $\alpha$ is a $(1,0)$-form. Since $d\theta=0,$ we  find  $\partial \alpha=0$ and 
$\bar\partial \alpha=-\partial \bar\alpha.$

\begin{defn}
A LCK manifold $(M, \omega)$ is Vaisman if $\nabla \theta=0,$ where $\nabla$ 
denotes the Levi-Civita connection associated to $\omega,$ and $\theta$ is the 
Lee form.
\end{defn}

\smallskip

We will recall next several results on Vaisman manifolds:

\begin{thm}
\label{collect}
Let $(M,\omega)$ be a compact Vaisman manifold. 
\begin{itemize}
\item[ 1)] (Vaisman, \cite{vaisman}) There exists a unique metric in its conformal class 
such that $|\theta|_\omega=1.$ 
\item[ 2)] (Verbitsky, \cite[Section 6]{vanishing}) If the Lee form $\theta$ satisfies  $|\theta|_\omega=1,$ 
then 
\begin{equation}
\label{omega}
\omega=2i\bar\partial\alpha+2i\alpha\wedge\bar \alpha,
\end{equation}
and all eigenvalues 
of the $(1,1)$-form $\bar\partial\alpha$ are positive, except one 
which is equal to zero\footnote{The eigenvalues of a $(1,1)$-form 
$\phi$ are the eigenvalues of the symmetric operator $G(\phi)$ defined by the
equation $\phi(v, Jv) = g(G(\phi)v, w),$ where $J$ is the ambient  
complex structure, and $g$ is a background Hermitian metric.}.
\end{itemize}
\end{thm}

\begin{proof}[Proof of Theorem \ref{vaisman}]  Let $M$ be a Vaisman manifold, 
with $\dim_\CC M=n\geq 3.$

\smallskip

The Bott-Chern cohomology of Vaisman manifolds was recently computed by 
Istrati and Otiman \cite{io}. In particular, from Theorem 4.2 in \cite{io}, 
and the duality between the Bott-Chern and Aeppli cohomologies 
$h^{0,1}_A=h^{n,n-1}_{BC}$ one can see that 
$$
h^{0,1}_A(M)=h^{0,1}_{BC}(M)+1.
$$

To prove that $M$ does not carry astheno-K\"ahler metrics, let $\omega$ be a LCK 
metric on $M,$ whose Lee form satisfies $|\theta|_\omega=1.$ 
Notice from Theorem \ref{collect}.2) that 
\begin{equation}
\label{agl-obs}
i\ddbar \omega=i\ddbar(2i\bar\partial\alpha+2i\alpha\wedge\bar \alpha)=
-2\bar\partial \alpha\wedge\partial \bar\alpha = 
-2(i\bar\partial \alpha)\wedge(i\bar\partial \alpha)\leq 0.
\end{equation}

Suppose now there exists an astheno-K\"ahler metric $\eta$  on $M.$ Using Stokes' 
theorem, we find:
\begin{equation*}
\label{std-argument}
0\geq \int_{M}\eta^{n-2} \wedge i\ddbar \omega 
=  \int_{M} i\ddbar \eta^{n-2} \wedge \omega=0
\end{equation*}
hence, $\ddbar\omega=0,$ which means $\omega$ is also a SKT metric. However, it was 
noticed by Alexandrov and Ivanov in \cite[Remark 1]{ai} (see also \cite[Theorem 1.3]{ip})  
that a  non-K\"ahler LCK metric of dimension at least $3$ cannot be SKT, which leads 
to a contradiction.
\end{proof}

\begin{rmk}
The authors are grateful to Alexandra Otiman for pointing out to them the relation 
between the Bott-Chern and Aeppli cohomology for $(0,1)$-forms in the case of 
Vaisman manifolds, and that Theorem \ref{vaisman} was also independently noticed 
by Angella, Otiman and Stanciu (unpublished). 
\end{rmk}

\begin{rmk}
The key observation (\ref{agl-obs}) in the above proof has also been recently 
noticed by Angella, Guedj and Lu \cite[Proposition 3.10]{agl} within the framework of 
LCK metrics with potential.
\end{rmk}

\begin{ex}
\label{lyz-conj}
Let $A\in GL(n,\CC)$ be a linear operator acting on $\CC^n$ with all eigenvalues 
$\lambda_i$ satisfying $|\lambda_i|>1.$ Denote by 
$\langle A\rangle \subseteq GL(n,\CC)$ 
the cyclic group generated by $A.$ The quotient 
$M=(\CC^n\setminus \{0\})/\langle A\rangle$ 
is called a linear Hopf manifold. If $A$ is diagonalizable, then $M$ is called a 
diagonal Hopf manifold. In \cite{ko}, Kamishima and Ornea proved that a diagonal 
Hopf manifold is Vaisman (see also \cite{ov}). As a consequence of Theorem
 \ref{vaisman}, no diagonal 
Hopf manifold $M$ carries an astheno-K\"ahler metric. This result can be used to 
confirm a conjecture of Li, Yau and Zheng regarding the class of similarity Hopf 
manifolds.
\begin{defn}
A compact complex manifold $X$ is called similarity Hopf manifold if and only if 
it is a finite undercover of a Hopf manifold 
$M=\left(\CC^n\setminus\{0\}\right)/\langle \phi\rangle,$ 
where $\phi(z) = azA, A\in U(n), {\bf z} = (z_1,\dots,z_n),$ and $a>1.$
\end{defn}
Li, Yau and Zheng conjectured in \cite[page 108]{lyz} that similarity Hopf manifolds 
cannot carry astheno-K\"ahler metrics. We obtain:
\begin{cor}
\label{lyz-answer}
There exists no astheno-K\"ahler metrics on similarity Hopf manifolds of dimension 
at least three.
\end{cor}
\begin{proof}
Suppose there exists a similarity Hopf manifold $X,\, \dim_\CC X= n\geq 3,$  
equipped with an astheno-K\"ahler metric $\omega.$ Let  
$
\pi:M\ra X
$ 
be an unramified finite covering, where 
$M=\left(\CC^n\setminus\{0\}\right)/\langle \phi\rangle,$ 
$\phi({\bf z}) = a{\bf z}A, A\in U(n), {\bf z} = (z_1,\dots,z_n),$ and $a>1.$ Notice that we have 
$$
\ddbar (\pi^*\omega)^{n-2}=\ddbar \pi^*\omega^{n-2}=\pi^*\ddbar\omega^{n-2}=0.
$$ 
That means $\pi^*\omega$ is an astheno-K\"ahler metric on the manifold $M.$ However, 
since the matrix $A$ is unitary, the matrix $aA$ is diagonalizable and for every eigenvalue  
$\lambda$ we have $|\lambda|=a>1.$ Therefore $M$ is a diagonal Hopf manifold, and 
such manifolds cannot carry astheno-K\"ahler metrics. 
\end{proof}

\end{ex}

\begin{rmk}
The standard Hopf manifold $M=\left(\CC^n\setminus\{0\}\right)/\langle \phi\rangle,$ 
where $\phi:\CC^n\ra\CC^n$  is a homothethy $\phi({\bf z})=a{\bf z},$ 
with $a\in \CC^*,\, |a|\neq 1,$   is an elliptic fiber bundle over $\CC\PP^{n-1}$ diffeomorphic 
to $S^1\times S^{2n-1}.$ As an example of a Vaisman 
manifold it does not carry astheno-K\"ahler metrics. It is interesting to notice that the 
closely related Calabi-Eckmann manifold, which is also an elliptic fiber bundle over 
$\CC\PP^{n-1}\times\CC\PP^{m-1}$ diffeomorphic to $S^{2n-1}\times S^{2m-1}$ 
admits astheno-K\"ahler metrics \cite{matsuo}, while satisfying $h^{0,1}_A=1$ and 
$h^{0,1}_{BC}=0$  for $1<n<m$ \cite{stelzig}.
\end{rmk}

We will conclude this paper by discussing  the saturation of the 
lower bound in (\ref{ineq}).

%%%%%%%%%%%%%%%%%%%%%%%%%%%%%%%%%%%%%%%%%%
%%%%%%%%%%%%%%%%%%%%%%%%%%%%%%%%%%%%%%%%%%
%%%%%%%%%%%%%%%%%%%%%%%%%%%%%%%%%%%%%%%%%%

 \subsection{A higher dimensional generalization of the completely solvable 
 Nakamura manifold}  
 
Our next example was first studied by Kasuya \cite{kasuya-hodge}. It generalizes 
I. Nakamura's example \cite[page 90]{nak} of a non-K\"ahler solvmanifold satisfying the 
$\ddbar$-lemma.

 \smallskip

Let $G=\CC\ltimes_\phi \CC^{2n},$ where 
$$
\phi(x+iy)(w_1,\dots,w_{2n})=(e^{a_1x}w_1,e^{-a_1x}w_2,
\dots,e^{a_nx}w_{2n-1},e^{-a_nx}w_{2n}),
$$  
where $a_i\neq0$ are integers. Notice we can write $G=\RR\times(\RR\ltimes_\phi \CC^{2n}).$ 
The group $G$ admits a co-compact lattice $\Gamma=t\ZZ\times \Delta,$  where $\Delta$ is a 
lattice in $\RR\ltimes_\phi \CC^{2n}$ for $t>0.$ For $t\neq r\pi$ for some $r\in \QQ,$ the complex 
manifold satisfies the Hodge symmetry and decomposition, but  $X=\Gamma\backslash G$ 
does not admits a K\"ahler metric \cite{kasuya-hodge}. In particular, we have 
$h^{0,1}_A(X)=h^{0,1}_{BC}(X).$ The interested reader may follow up the arguments in 
\cite{ak} to find that $h^{0,1}_A(X)=h^{0,1}_{BC}(X)=1.$

Consider now the $(1,1)$-form 
$$
\eta=\sqrt{-1}\left(dz\wedge d\bar z
+\sum_{i=1}^n (e^{-2a_ix}dw_{2i-1}\wedge d\bar w_{2i-1} 
+e^{2a_ix}dw_{2i}\wedge d\bar w_{2i})\right),
$$
on $X$ induced from $\CC\ltimes_\phi \CC^{2n}.$ Then one can see that $i\ddbar \eta$ 
is a weakly positive, $\ddbar$-exact non-vanishing $(2,2)$-current. By Theorem \ref{hl-obs}, 
it follows that $X$ cannot carry any astheno-K\"ahler metric.
 
 \begin{rmk}
 In \cite{fkv}, Fino, Kasuya and Vezzoni proved that  Kasuya's example above 
 cannot carry SKT metrics as well.
 \end{rmk}

%%%%%%%%%%%%%%%%%%%%%%%%%%%%%%%%%%%%%%%%%%
%%%%%%%%%%%%%%%%%%%%%%%%%%%%%%%%%%%%%%%%%%
%%%%%%%%%%%%%%%%%%%%%%%%%%%%%%%%%%%%%%%%%%

 \subsection{Fujiki class $\mathcal C$ manifolds of complex dimension three.} 
 
Manifolds satisfying the $\ddbar-$lemma automatically  satisfy $h^{0,1}_A=h^{0,1}_{BC}.$ 
In particular, manifolds of Fujiki class $\mathcal C,$ i.e., manifolds bimeromorphic to K\"ahler 
manifolds, satisfy such condition. In \cite{chiose}, the first author proved that a Fujiki class 
$\mathcal C$ manifold admits a SKT metric if and only if it is of K\"ahler type. In particular, 
any Fujiki class $\mathcal C$ manifold of complex dimension three which admits an 
astheno-K\"ahler metric is of K\"ahler type. As a consequence, we notice:

\begin{cor}
\label{modifications}
The class of astheno-K\"ahler manifolds is not invariant under modifications.
\end{cor}
\proof
Let $M$ be the $3$ dimensional manifold constructed by Hironaka \cite{hironaka} 
which is a proper modification of the projective space $\CC\PP^3$ and which contains 
a positive linear combination of curves which is homologuous to $0.$ In particular, 
$M$ is a non-K\"ahler Fujiki class $\mathcal C$ manifold, which, unlike $\CC\PP^3,$ 
cannot not carry an astheno-K\"ahler metric.
\qed

\begin{rmk}
The authors are not aware of an example of a non-K\"ahler manifold admitting 
an astheno-K\"ahler metric and satisfying $h^{0,1}_A(M)=h^{0,1}_{BC}(M)$. 
\end{rmk}

%%%%%%%%%%%%%%%%%%%%%%%%%%%%%%%%%%%%%%%%%%
%%%%%%%%%%%%%%%%%%%%%%%%%%%%%%%%%%%%%%%%%%
%%%%%%%%%%%%%%%%%%%%%%%%%%%%%%%%%%%%%%%%%%

\subsection*{Acknowledgements} The second author was partially supported by a 
Professional Travel Grant from Vanderbilt University. The authors would like to thank 
Gueo Grantcharov for suggestions and encouragements in the early stages of this work, 
Alexandra Otiman for helping them understand the cohomology of Oeljeklaus-Toma and 
Vaisman manifolds and for useful comments, and Victor Vuletescu for clarifications regarding 
locally K\"ahler manifolds.

%%%%%%%%%%%%%%%%%%%%%%%%%%%%%%%%%%%%%%%%%%
%%%%%%%%%%%%%%%%%%%%%%%%%%%%%%%%%%%%%%%%%%
%%%%%%%%%%%%%%%%%%%%%%%%%%%%%%%%%%%%%%%%%%

\providecommand{\bysame}{\leavevmode\hbox
to3em{\hrulefill}\thinspace}


\begin{thebibliography}{Dyn52b}



\bibitem[AKMW]{wft}
{D. Abramovich, K. Karu, K. Matsuki, and J. W\l odarczyk}, 
{\em Torification and factorization of birational maps}, 
J. Amer. Math. Soc. {\bf 15} (2002), no.~3, 531--572.



\bibitem[AI]{ai}
{B. Alexandrov, S. Ivanov,} 
{\em Vanishing theorems on Hermitian manifolds.} 
Differential Geom. Appl. {\bf 14} (2001), no. 3, 251--265.



\bibitem[An]{angella}
{ D. Angella,}
{\em The cohomologies of the Iwasawa manifold and of its small deformations.}
J. Geom. Anal. {\bf 23} (2013), no. 3, 1355--1378. 



\bibitem[ADOS]{ados}
{D. Angella, A. Dubickas, A. Otiman, J. Stelzig,} 
{\em On metric and cohomological properties of Oeljeklaus-Toma manifolds.}	
\href{https://arxiv.org/abs/2201.06377}{arXiv:2201.06377 [math.DG]}.



\bibitem[AGL]{agl}
{ D. Angella, V. Guedj, C. C. Lu}
{\em Plurisigned hermitian metrics.}
\href{https://arxiv.org/abs/2207.04705}{arXiv:2207.04705v1 [math.CV]}.



\bibitem[AK]{ak}
{D. Angella, H. Kasuya,} 
{\em Bott-Chern cohomology of solvmanifolds.} 
Ann. Global Anal. Geom. {\bf 52} (2017), no. 4, 363--411.



\bibitem[APV]{apv}
{D. Angella, M. Parton, V. Vuletescu,} 
{\it Rigidity of Oeljeklaus-Toma manifolds,.} 
Ann. Inst. Fourier {\bf 70} (2020), no. 6, 2409--2423.



\bibitem[Bi]{bismut}
{J. M. Bismut,} 
{\em A local index theorem of non-K\"ahler manifolds,} 
Math. Ann. {\bf 284} (1989) 681-- 699.



\bibitem[Bis]{biswas} 
{I. Biswas,} 
{\em On Hermitian structure on $G/\Gamma.$} 
Bull. Sci. Math. {\bf 137} (2013), no. 6, 716--717. 



\bibitem[CT]{ct}
{J. A.Carlson, D. Toledo,} 
{\em On fundamental groups of class VII surfaces.} 
Bull. London Math. Soc. {\bf 29} (1997): 98--102.



\bibitem[Ch]{chiose}
{I. Chiose,} 
{\em Obstruction to the existence of K\"ahler structures on compact complex manifolds.} 
Proc. Amer. Math. Soc. {\bf 142}, 3561--3568 (2014).



\bibitem[CFGU]{cfgu}
{L. A. Cordero, M. Fern\'andez, A. Gray, L. Ugarte,} 
{\em Compact nilmanifolds with nilpotent complex structure: Dolbeault cohomology.} 
Trans. Amer. Math. Soc. {\bf 352} (2000) 5405--5433.



\bibitem[DGMS]{dgms}
{P. Deligne, P. Griffiths, J. Morgan, D. Sullivan,} 
{\em Real homotopy theory of K\"ahler manifolds.} 
Invent. Math. {\bf 29} (1975), no. 3, 245--274.



\bibitem[FPS]{fps}
{A. Fino, M. Parton, S. Salamon,} 
{\em Families of strong KT structures in six dimensions.} 
Comment. Math. Helv. {\bf 79} (2), (2004), 317--340.



\bibitem[FGV]{fgv}
{A. Fino, G. Grantcharov, L. Vezzoni,} 
{\em Astheno-K\"ahler and balanced structures on fibrations.} 
Int. Math. Res. Not. {\bf IMRN 2019}, no. 22, 7093--7117. 



\bibitem[FT1]{ft1}
{A. Fino, A. Tomassini,}  
{\em Blow-ups and resolutions of strong K\"ahler with torsion metrics}, 
Adv. Math. {\bf 221} (2009) 914--935.



\bibitem[FT2]{ft2}
{A. Fino, A. Tomassini,} 
{\em On astheno-K\"ahler metrics.} 
J. Lond. Math. Soc. (2) {\bf 83} (2) (2011) 290--308.



\bibitem[FKV]{fkv}
{A. Fino, H. Kasuya, L. Vezzoni,} 
{\em SKT and tamed symplectic structures on solvmanifolds.} 
Tohoku Math. J. (2) {\bf 67} (2015), no. 1, 19--37.
 


\bibitem[FWW]{fww}
{J. Fu, Z. Wang, D. Wu,} 
{\em Semilinear equations, the $\gamma_k$-function, and generalized Gauduchon metrics,} 
J.Eur. Math. Soc., {\bf 15}, 2013, pp. 659--680.



\bibitem[Ga]{gauduchon}
{P. Gauduchon}, 
{\em Le th\'eor\`eme de l'excentricit\'e nulle.}
C.R Acad. Sci. Paris S\'er. A-B 2{\bf 285} (1977), no.5, A387--A390.



\bibitem[GH]{gh}
{P. Griffiths, J. Harris,} 
{\em Principles of algebraic geometry.} Reprint of the 1978 original. 
Wiley Classics Library. John Wiley \& Sons, Inc., New York, 1994. 



\bibitem[HL]{hl}
{ R. Harvey, H.B. Lawson,} 
{\em An intrinsic characterization of K\"ahler manifolds.} 
Invent. Math. {\bf 74} (1983), no. 2, 169--198.



\bibitem[Hi]{hironaka}
{\em H. Hironaka,} 
{\em An example of a non-K\"ahlerian complex-analytic deformation of K\"ahlerian complex structures.} 
Ann. of Math. (2) {\bf 75} (1962), 190--208.



\bibitem[In]{inoue}
{M. Inoue,} 
{\em On surfaces of class $VII_0.$}
Invent. Math., {\bf 24} (1974), 269--310.



\bibitem[IO]{io}
{N. Istrati, A. Otiman,}
{\em Bott-Chern cohomology of compact Vaisman manifolds.}
\href{https://arxiv.org/pdf/2206.07312}{arXiv:2206.07312v2 [math.DG]}, 
to appear in Trans. Amer. Math. Soc.



\bibitem[IP]{ip}
{S. Ivanov, G. Papadopoulos,} 
{\em Vanishing theorems on $(\ell\vert k)$-strong K\"ahler manifolds with torsion.} 
Adv. Math. {\bf 237} (2013), 147--164.



\bibitem[JY]{jy}
{J. Jost,  S-T. Yau.} 
{\em A non-linear elliptic system for maps from Hermitian to Riemannian
manifolds and rigidity theorems in Hermitian geometry.} 
Acta Math. {\bf 170} (1993), 221--54;
Corrigendum Acta Math. {\bf 173} (1994), 307.



\bibitem[KO]{ko}
{ Y. Kamishima, L. Ornea,} 
{\em Geometric flow on compact locally conformally K\"ahler manifolds.}
Tohoku Math. J. (2) {\bf 57} (2005), no. 2, 201--221.



\bibitem[Ka1]{kasuya}
{H.  Kasuya,}
{\em  Vaisman metrics on solvmanifolds and Oeljeklaus-Toma manifolds.} 
Bull. Lond. Math. Soc. {\bf 45} (2013), no. 1, p. 15--26.




\bibitem[Ka2]{kasuya-hodge}
{H.  Kasuya,} 
{\em Hodge symmetry and decomposition on non-K\"ahler solvmanifolds.} 
J. Geom. Phys. {\bf 76} (2014), 61--65. 



\bibitem[Ko]{kodaira}
{K. Kodaira,} 
{\em On the Structure of Compact Complex Analytic Surfaces, II, III} 
 Amer. J. Math. {\bf 88} (1966), 682--721, Amer. J. Math, {\bf 90} (1968), 55--83.



\bibitem[LU]{lu}
{A. Latorre,  L. Ugarte,} 
{\em On non-K\"ahler compact complex manifolds with balanced and astheno-K\"ahler metrics.} 
C. R. Math. Acad. Sci. Paris {\bf 355} (2017), no. 1, 90--93. 



\bibitem[LUV]{luv}
{A. Latorre,  L. Ugarte, R. Villacampa,} 
{\em On the Bott-Chern cohomology and balanced Hermitian nilmanifolds.} 
Internat. J. Math. {\bf 25} (2014), no. 6, 1450057, 24 pp.



\bibitem[LYZ]{lyz}
{J. Li, S.T. Yau, F. Zheng,} 
{\em On projectively flat Hermitian manifolds,} 
Comm. Anal. Geom. {\bf 2} (1994), no. 1, 103--109. 



\bibitem[Mal]{malcev}
{A. I. Malcev,}
{\em On a class of homogeneous spaces,} 
Amer. Math. Soc. Translation Ser. 1, {\bf 9} (1962), 276--307.



\bibitem[Mat]{matsuo}
{K. Matsuo,} 
{\em Astheno-K\"ahler structures on Calabi-Eckmann manifolds.} 
Colloq. Math. {\bf 115}, no. 4 (2009), 33--99.



\bibitem[Na]{nak}
{I. Nakamura,} 
{\em Complex parallelisable manifolds and their small deformations.} 
J. Differ. Geom. {\bf 10} (1) (1975), 85--112.



\bibitem[No]{no}
{K. Nomizu,} 
{\em On the cohomology of compact homogeneous spaces of nilpotent Lie groups,}
Ann. of Math. (2) {\bf 59} (1954), 531--538.



\bibitem[OT]{ot}
{K. Oeljeklaus,  M. Toma,}
{\em Non-K\"ahler compact complex manifolds associated to number fields.} 
Ann. Inst. Fourier {\bf 55} (2005), no. 1, 161--171. 



\bibitem[OV]{ov}
{L. Ornea, M. Verbitsky,} 
{\em Locally conformally K\"ahler metrics obtained from pseudoconvex shells.} 
Proc. Amer. Math. Soc. {\bf 144} (2016), no. 1, 325--335.





\bibitem[RT]{rt}
{F. A. Rossi, A. Tomassini,}
{\em On strong K\"ahler and astheno-K\"ahler metrics on nilmanifolds.}
Adv. Geom. {\bf 12} (2012), no. 3, 431--446. 



\bibitem[Sa]{sal}
{S. Salamon,} 
{\em Complex structures on nilpotent Lie algebras,} 
J. Pure Appl. Algebra {\bf 157} (2001) 311--333.



\bibitem[Sc]{sch} 
{M. Schweitzer}, 
{\em Autour de la cohomologie de Bott-Chern.} 
\href{https://arxiv.org/abs/0709.3528}{arXiv:0709.3528 [math.AG]}.



\bibitem[ST]{st}
{T. Sferruzza, A. Tomassini,} 
{\em On cohomological and formal properties of Strong K\"hler with torsion and astheno-K\"ahler metrics.} 
\href{https://arxiv.org/abs/2206.06904}{arXiv:2206.06904 [math.DG]}.



\bibitem[St1]{stelzig-blow}
{J. Stelzig,} 
{\em The double complex of a blow-up.} 
Int. Math. Res. Not. IMRN {\bf 2021}, no. 14, 10731--10744.



\bibitem[St2]{stelzig}
{J. Stelzig,} 
{\em On the structure of double complexes.} 
J. Lond. Math. Soc. (2) {\bf 104} (2021), no. 2, 956--988.




\bibitem[Sw]{swann}
{A. Swann,}
{\em Twisting Hermitian and hypercomplex geometries.} 
Duke Math. J. {\bf 155} (2010), 403--431.




\bibitem[Ug]{ugarte}
{L. Ugarte,} 
{\em Hermitian structures on six-dimensional nilmanifolds.}
 Transform. Groups {\bf 12} (2007), no. 1, 175--202. 




\bibitem[Va]{vaisman}
{I.Vaisman,}
{\em Generalized Hopf manifolds.} 
Geom. Dedicata {\bf 13} (1982), no. 3, 231--255.



\bibitem[Ve1]{vanishing}
{M. Verbitsky}
{\em Theorems on the vanishing of cohomology for locally conformally hyper-K\"ahler manifolds.}
Tr. Mat. Inst. Steklova {\bf 246} (2004), Algebr. Geom. Metody, Svyazi i Prilozh., 64--91; translation in 
Proc. Steklov Inst. Math. 2004, no. 3 (246), 54--78, 
\href{https://arxiv.org/abs/math/0302219}{arXiv:math/0302219v4 [math.DG]}.



\bibitem[Ve2]{verbitsky}
{M. Verbitsky},
{\em Rational curves and special metrics on twistor spaces.}
 Geom. Topol. {\bf 18} (2014), no. 2, 897--909.
 
 
 
 \bibitem[Vo]{voisin}
{C. Voisin,} 
{\em Hodge Theory and Complex Algebraic Geometry. I.} 
Cambridge Studies in Advanced Mathematics, 76, Cambridge University Press, Cambridge, 2002.         



 
\end{thebibliography}
\end{document}